\documentclass[onefignum,onetabnum]{siamonline220329} %

\usepackage{lipsum}
\usepackage{amsfonts}
\usepackage{graphicx}
\usepackage{epstopdf}
\usepackage{algorithmic}
\ifpdf
  \DeclareGraphicsExtensions{.eps,.pdf,.png,.jpg}
\else
  \DeclareGraphicsExtensions{.eps}
\fi

\usepackage{enumitem}
\setlist[enumerate]{leftmargin=.5in}
\setlist[itemize]{leftmargin=.5in}

\newsiamremark{remark}{Remark}
\newsiamremark{hypothesis}{Hypothesis}
\crefname{hypothesis}{Hypothesis}{Hypotheses}
\newsiamthm{claim}{Claim}

\headers{Efficient Diffusion Posterior Sampling for Noisy Inverse Problems}{Ji Li, and Chao Wang}

\title{Efficient Diffusion Posterior Sampling for Noisy Inverse Problems\thanks{Submitted to the editors DATE.
\funding{This work was funded by 2024 Youth Talents Support Program from Capital Normal University and supported by the Open Project of Key Laboratory of Mathematics and Information Networks (Beijing University of Posts and Telecommunications), Ministry of Education, China, under Grant No. KF202401).}}}

\author{Ji Li\thanks{Academy for Multidisciplinary Studies, Capital Normal University, Beijing 
  (\email{matliji@163.com}).}
\and Chao Wang\thanks{Department of Radiation Oncology, University of Kansas Medical Center, Kansas City
  (\email{wywwwnx@163.com}).}}

\usepackage{amsopn}
\DeclareMathOperator{\diag}{diag}

\newcommand{\norm}[1]{\left\lVert#1\right\rVert}

\usepackage{algorithm}
\usepackage{algorithmic}

\usepackage{booktabs}
\usepackage{multirow}
\usepackage{mathtools}
\usepackage{bm}
\usepackage{url}
\usepackage{hyperref}
\usepackage{eso-pic}
\usepackage{xspace}

\newlength\stextwidth
\newcommand\makesamewidth[3][c]{%
  \settowidth{\stextwidth}{#2}%
  \makebox[\stextwidth][#1]{#3}%
}

\makeatletter
\DeclareRobustCommand\onedot{\futurelet\@let@token\@onedot}
\def\@onedot{\ifx\@let@token.\else.\null\fi\xspace}

\def\eg{\emph{e.g}\onedot} 
\def\ie{\emph{i.e}\onedot}

\def\etal{\emph{et al}\onedot}
\makeatother

\newcommand{\bl}{\color{blue}}
\newcommand{\ul}[1]{\underline{#1}}
\usepackage{multirow}

\begin{document}

\maketitle

\begin{abstract}
  The pretrained diffusion model as a strong prior has been leveraged to address inverse problems in a zero-shot manner without task-specific retraining. Different from the unconditional generation, the measurement-guided generation requires estimating the expectation of clean image given the current image and the measurement. With the theoretical expectation expression, the crucial task of solving inverse problems is to estimate the noisy likelihood function at the intermediate image sample. Using the Tweedie's formula and the known noise model, the existing diffusion posterior sampling methods perform gradient descent step with backpropagation through the pretrained diffusion model. To alleviate the costly computation and intensive memory consumption of the backpropagation, we propose an alternative maximum-a-posteriori (MAP)-based surrogate estimator to the expectation. With this approach and further density approximation, the MAP estimator for linear inverse problem is the solution to a traditional regularized optimization, of which the loss comprises of data fidelity term and the diffusion model related prior term. Integrating the MAP estimator into a general denoising diffusion implicit model (DDIM)-like sampler, we achieve the general solving framework for inverse problems. Our approach highly resembles the existing $\Pi$GDM without the manifold projection operation of the gradient descent direction. The developed method is also extended to nonlinear JPEG decompression. The performance of the proposed posterior sampling is validated across a series of inverse problems, where both VP and VE SDE-based pretrained diffusion models are taken into consideration.
\end{abstract}



\section{Introduction}
\label{sec:intro}

In the field of image reconstruction, the degradation of the measurement procedure can be modeled as a linear transformation of clear images. These problems include denoising, inpainting, image super-resolution, compressing sensing~\cite{donoho2006compressed}, computed tomography (CT) \cite{chakhlov2016current}, magnetic resonance imaging (MRI)~\cite{liang2000principles} in medicine, and many other problems.
For a variant of linear inverse problems, we have the measurement $\bm{y}$ from the forward model
\begin{equation}
  \label{eq:prob}
  \bm{y} = \bm{A}\bm{x}_0+\bm{n}, 
\end{equation}
where $\bm{A}$ is the linear operator of the model, and $\bm{n}$ is the measurement noise. Generally, we have $\bm{A}\in\mathbb{R}^{m\times n},m\leq n$ to imply the ill-posedness of the problem. In the case of Gaussian noise, $\bm{n}\sim \mathcal{N}(0,\sigma_y^2\bm{I})$. In the explicit form, the likelihood $p(\bm{y}|\bm{x}_0)\sim\mathcal{N}(\bm{y}|\bm{A}\bm{x}_0,\sigma_y^2\bm{I})$.

From the Bayesian image restoration perspective, we consider the posterior sampling from $p(\bm{x}_0|\bm{y})$, which is proportional to $p(\bm{y}|\bm{x}_0)p(\bm{x}_0)$, to achieve the solution to~\eqref{eq:prob}. The key component of solving inverse problems is to infer the distribution $p(\bm{x}_0)$ as the noise model $p(\bm{y}|\bm{x}_0)$ is assumed known. The prior $p(\bm{x}_0)$ to some extent helps regularize the problem and circumvent the noise amplification. The design of prior $p(\bm{x}_0)$ is an ever-lasting theme in the history of solving inverse problems. There are many approximated priors and related algorithms developed for inverse problems, including the hand-crafted sparsity promoting prior for a universal natural image, and the learned sparse representation under a dictionary basis.

The misspecified or inaccurate prior $p(\bm{x}_0)$ leads to performance degradation of image restoration. We expect that utilizing the ideal $p_{\text{data}}(\bm{x}_0)$ can boost performance significantly. Due to the intractability of the underlying $p_{\text{data}}(\bm{x}_0)$, we seek the best density estimation as a surrogate. In the light of outstanding generation ability, the diffusion model~\cite{song2020score,song2019generative,ho2020denoising} as a prior has been leveraged to solve inverse problems.

According to the usage manners of the diffusion models, the existing diffusion model-based approaches can be categorized into two groups: \emph{conditional score matching methods} that retrain a conditional score neural network using datasets consisting of paired degradation and clean images~\cite{saharia2022image,saharia2022palette,whang2022deblurring,luo2023image,xia2023diffir,gao2023implicit}, and the \emph{zero-shot methods} that leverage pretrained unconditional diffusion models to perform posterior sampling for variants of inverse problems without the task-specific retraining process. Compared to conditional score matching methods, the zero-shot methods have a distinct advantage by avoiding model retraining on task-specific datasets, thereby enhancing their applicability across a broader spectrum of applications. It beats the conditional score matching methods in the flexibility, as it decouples the data fidelity and the prior like the traditional solutions to inverse problems.

In this paper, we focus on the zero-shot methods for its avoidance of task-specific retraining. In order to guide the conditional generation process, the measurement information is fused in the intermediate iterations via several approaches. The first approach completed the null space inpainting in the spectral domain~\cite{choi2021ilvr,lugmayr2022repaint,kawar2022denoising,wang2022zero}. The second approach enforces data consistency to approximate the intermediate predicted clean image by solving a proximal optimization problem~\cite{song2021solving,zhu2023denoising}. The third approach estimates the conditional score function and leverages it to the unconditional sampling by Langevin dynamics~\cite{jalal2021robust,chung2022score,kawar2021snips,kawar2021stochastic} or modifying the unconditional reverse sampling followed by a gradient descent step~\cite{chung2022diffusion,meng2022diffusion,song2022pseudoinverse,feng2023score}. Lastly, some works adopted the variational inference to provide the distribution surrogate for the true posterior distribution~\cite{graikos2022diffusion,feng2023score,ozturkler2023regularization}. Overall, all the methods can be regarded as approximations to the posterior sampling from $p(\bm{x}_0|\bm{y})$. Notice that the methods~\cite{kawar2022denoising,song2022pseudoinverse,wang2022zero} depend on the available SVD for the linear operator. The restriction narrows their applications for linear problems without available SVD, such as the tomography in medicine~\cite{wang2022zero,kawar2022denoising,kawar2021snips}. We focus on the SVD-free approximated posterior sampling methods for broader applications. 

Notice that there are two aspects of posterior diffusion sampling methods for inverse problem: the generation sampler and the measurement guidance approach. To the best of our knowledge, there is a lack of the solving framework taking the two aspects into consideration. To address the issue, we introduce a parameterized non-Markov chain based DDIM-like sampler with tunable magnitude of injected noise per iteration and a general measurement-guided conditional generation scheme. The DDIM-like sampler is able to improve the restoration performance with tunable parameter. Recognizing the unconditional generation only depends on the expectation $\mathbb{E}[\bm{x}_0|\bm{x}_t]$, the conditional generation is achieved with $\mathbb{E}[\bm{x}_0|\bm{x}_t]$ being replaced by expectation $\mathbb{E}[\bm{x}_0|\bm{x}_t,\bm{y}]$, which depends on $\nabla_{\bm{x}_t}\log p(\bm{y}|\bm{x}_t)$. With this view, the existing methods~\cite{chung2022diffusion,song2022pseudoinverse} estimated $p(\bm{y}|\bm{x}_t)$ using different estimations of $p(\bm{x}_0|\bm{x}_t)$, which equivalently estimated $\mathbb{E}[\bm{x}_0|\bm{x}_t,\bm{y}]$. The computation of $\nabla_{\bm{x}_t}\log p(\bm{y}|\bm{x}_t)$ will require the backpropagation through the neural score network, which is computationally demanding and costly for the memory consumption. More details are discussed in Section~\ref{sec:ex}. To avoid such backpropagation, instead of using the Tweedie's formula, DMPS~\cite{meng2022diffusion} resorts to the Gaussian $p(\bm{x}_0|\bm{x}_t)=\mathcal{N}(\gamma_t\bm{x}_t,\bar{r}_t^2\bm{I})$ with a crude uninformative assumption of $p(\bm{x}_0)$.

To summarize, our contributions are as follows:
\begin{enumerate}
\item We propose a unified posterior diffusion sampling framework for inverse problem with parameterized DDIM-like sampler and expectation $\mathbb{E}[\bm{x}_0|\bm{x}_t,\bm{y}]$-based measurement guidance.
\item To save computational cost of existing approximations to $\mathbb{E}[\bm{x}_0|\bm{x}_t,\bm{y}]$, we propose the maximum-a-posteriori (MAP) surrogate for $\mathbb{E}[\bm{x}_0|\bm{x}_t,\bm{y}]$, which leads to a novel and efficient measurement guidance. The derived method is also highly related to the existing works~\cite{choi2021ilvr,chung2022come,song2021solving} for noiseless measurements, ignoring the possible scaling. The key advantages of our estimation are the SVD-free and backpropagation free properties. As a result, our method results in a notable acceleration of the overall computational runtime than the existing DPS and $\Pi$GDM.
\item We extend the method to some specific noisy nonlinear problems, such as JPEG decompressing. 
\item We validate our proposed scheme across a variety of linear problems, demonstrating the effectiveness of our method.
\end{enumerate}

The organization of the paper is as follows. In Section~\ref{sec:review}, we review the related work for solving inverse problems, especially the generative models for inverse problems. In Section~\ref{sec:back}, we present the background of diffusion models, the posterior sampling framework for inverse problem solving and the existing works on diffusion posterior sampling for linear problems. In Section~\ref{sec:meth}, the main methodology of the paper is provided. In Section~\ref{sec:exp}, our method is validated across a series of linear problems and some specific nonlinear problems. The conclusion comes the last in Section~\ref{sec:conclusion}.

\section{Related work}\label{sec:review}

\subsection{Traditional hand-crafted prior for inverse problems}
Various approximated prior for natural images are investigated, including the sparsity-inspired statistics~\cite{rudin1992nonlinear,chambolle2010introduction,chan2006total,cai2012image,dong2012wavelet} and other feature occurrence prior~\cite{pleschberger2020explicit}. The hand-crafted prior-based solution works well for a wide category of image restoration tasks, but their performance is upper limited due to the lack of a more specific statistics description for certain images. Learning-based dictionary biases~\cite{aharon2006k}, or data-driven wavelet frames~\cite{quan2015data} can be constructed such that the restored image has a sparse expression within the basis. However, the dictionary basis is image-specific, which results in the necessity of re-learning the dictionary when applied to a different image. 

\subsection{Supervise learning for inverse problems}
With the emergence of deep learning techniques, many works directly learn the solution map from measurement $\bm{y}$ to the estimated $\bm{x}_0$ using supervised learning, see \eg \cite{Mousavi2015,zhang2017learning,liu2018image,zhang2018ista,sun2016deep,shi2019scalable,Nan_2020_CVPR,ding2020low}. The learning is driven by a collection of paired datasets, and the model is trained by individual image task. Therefore, its capacity for generalization in diverse image tasks is limited, even in the absence of varying measurement noise, not to mention scenarios of changing the degradation model. 
Supervised learning does not explicitly learn the prior $p_{\text{data}}(\bm{x}_0)$. The trained model is task-specific, limiting its adaptability for transferring to other closely related tasks. 

\subsection{GAN inversion and conditional normalizing flow}
There are several kinds of generative models, including variational auto-encoder (VAE)~\cite{kingma2013auto}, adversarial generative network (GAN)~\cite{goodfellow2020generative}, normalizing flow (NF)~\cite{kingma2018glow}. Exploiting of GAN and NF in solving inverse problems can be referred to~\cite{bora2017compressed,whang2021solving}. The so-called \emph{GAN inversion} assumes that the prior on the generative image is imposed by the network manifold. The NF actually gives the evaluation of $-\log p(\bm{x}_0)$~\cite{kingma2018glow,dinh2016density}, and it can serve as a regularization term to solve inverse problems. The application to compressed sensing for VAE~\cite{kadkhodaie2021stochastic} and GAN models~\cite{bora2017compressed,pan2021exploiting} can be found in the literature. Due to the representation errors, their application to inverse problems does not perform well as the diffusion model as a prior for inverse problems.

\subsection{Diffusion models for inverse problems}
As a prior, diffusion posterior sampling is the main framework for inverse problems. DDRM~\cite{kawar2022denoising} operators the information fusion in the spectral domain. DDNM~\cite{wang2022zero} resorts to the null space decomposition in the clean image domain. Both of them are SVD-dependent, as the implemented algorithm considers iteration in the spectral domain. Other works estimate the unknown likelihood $p(\bm{y}|\bm{x}_t)$ at the intermediate step. Several assumptions and estimations are proposed, such as~\cite{chung2022diffusion,song2022pseudoinverse,meng2022diffusion}.

\subsection{Direct diffusion models for inverse problems}
Like supervised learning, direct diffusion model trained a deep model, which is leveraged to map the measurement to the solution via an iterative scheme. The training stage borrows the Schr\"odinger bridge concept~\cite{pmlr-v202-liu23ai} or employs the new degradation SDE involving the measurement~\cite{delbracio2023inversion,luo2023image}.

\subsection{Plug-and-play methods for inverse problems}
Plug-and-play methods build upon the insight that the proximal step of solution to regularized optimization for inverse problems is effectively a denoising operation. Hence, such step can be performed using non-learned denoisers, such as BM3D, or the neural network-based denoiser~\cite{zhang2017beyond}. Along with the PnP insight, many PnP-based methods have been developed, such as~\cite{meinhardt2017learning,sun2019online,tan2024provably}. The convergence guarantees of such PnP-based method assume non-expansiveness of the denoiser~\cite{ryu2019plug,hurault2022gradient,hurault2022proximal}. Recently, the flow-based generative model is also integrated into the PnP-based method~\cite{martin2024pnp}.

\section{Background}\label{sec:back}

\subsection{SDE modeling for diffusion models}

The diffusion model is a powerful approach to generate a clean image from the underlying data distribution of a given dataset. The diffusion model comprises of two processes: \emph{forward process} and \emph{generative process}. The procedure of the forward process is to gradually perturb the clean data $\bm{x}_0$ with a series of preconfigured noise scales. From the continuous modeling perspective, the process is described using the stochastic differential equation (SDE):
\begin{equation}
  \label{eq:for}
  d\bm{x}_t = f(t)\bm{x}_tdt + g(t) d\bm{w}_t,
\end{equation}
where $\bm{w}_t$ is the standard Brownian process. Solving~\eqref{eq:for}, we can derive the transition conditional distribution
\begin{equation}
  \label{eq:trans}
  q(\bm{x}_t|\bm{x}_0)\sim \mathcal{N}(\bm{x}_t|\alpha_t\bm{x}_0,\sigma_t^2\bm{I}),\quad \text{ for }t\in (0,T],
\end{equation}
where $\alpha_t,\sigma_t$ are functions of $t$.
More specifically, we have
\begin{equation}
    \alpha_t = \exp(\int_0^t f(\tau) d\tau), \qquad \sigma_t^2= \int_0^t\exp(\int_{\tau}^t 2f(s)ds)g^2(\tau)d\tau.
\end{equation}
The options of $f(t),g(t)$ shall ensure that $\bm{x}_T$ is a standard Gaussian $\mathcal{N}(\bm{0},\bm{I})$. %

The generative process is to revert the forward process, it obtains a clean image gradually from a Gaussian noise image $\bm{x}_T$. The reverse process is governed by the reverse SDE:
\begin{equation}
  \label{eq:rev}
  d\bm{x}_t = \left[f(t)\bm{x}_t-g^2(t)\nabla_{\bm{x}_t}\log p(\bm{x}_t)\right]dt + g(t) d\overline{\bm{w}}_t,
\end{equation}
where $d\overline{\bm{w}}_t$ corresponds the Brownian process running in backward time. To enable the generation, the unknown score function $\nabla_{\bm{x}_t}\log p(\bm{x}_t)$ is approximated by a time-dependent neural network $s_{\theta}(\bm{x}_t,t)$, which is trained by minimizing the score matching loss
\begin{equation}
  \label{eq:sm}
  \mathcal{L}(\theta):=\mathbb{E}_t\left\{\lambda(t)\mathbb{E}_{\bm{x}_0}\mathbb{E}_{\bm{x}_t|\bm{x}_0}\left[\norm{s_{\theta}(\bm{x}_t,t)-\nabla_{\bm{x}_t}\log p(\bm{x}_t|\bm{x}_0)}_2^2\right]\right\},
\end{equation}
where $\lambda(t)$ is a weighting function~\cite{song2021maximum}.

\subsection{Inverse problems using diffusion models}

In the Bayesian framework, the solution to problem~\eqref{eq:prob} is to produce samples from the conditional distribution $p(\bm{x}_0|\bm{y})$. In lieu of the role of score function $\nabla_{\bm{x}_t}\log p(\bm{x}_t)$ in~\eqref{eq:rev}, the samples from $p(\bm{x}_0|\bm{y})$ can be generated by solving the following SDE
\begin{equation}
  \label{eq:uncond}
    d\bm{x}_t = \left[f(t)\bm{x}_t-g^2(t)\nabla_{\bm{x}_t}\log p(\bm{x}_t|\bm{y})\right]dt + g(t) d\overline{\bm{w}}_t.
  \end{equation}
The unknown \emph{conditional score} function $\nabla_{\bm{x}_t}\log p(\bm{x}_t|\bm{y})$ requires approximation. Similar to the score function $\nabla_{\bm{x}_t}\log p(\bm{x}_t)$, some works learns a deep model $\bm{s}_{\theta}(\bm{x}_t,t;\bm{y})$ using paired dataset~\cite{whang2022deblurring,liu2023dolce}. The measurement $\bm{y}$, as the conditional information, is also fed into in the network. This approach requires task-specific training the neural network for approximating score function. Hence, when the image restoration task is shift, the retraining of the conditional score function is required. Another retraining-free approach is to utilize the Bayes' rule. The target $p(\bm{x}_t|\bm{y})$ connects the prior $p(\bm{x}_t)$ and the likelihood function $p(\bm{y}|\bm{x}_t)$. To leverage the underlying prior from the diffusion model, we can achieve the samples by solving the reverse SDE:
\begin{equation}
  \label{eq:cond_sde}
    d\bm{x}_t = \left[f(t)\bm{x}_t-g^2(t)(\nabla_{\bm{x}_t}\log p(\bm{x}_t)+\nabla_{\bm{x}_t}\log p(\bm{y}|\bm{x}_t))\right]dt + g(t) d\overline{\bm{w}}_t.
\end{equation}
Consequently, we can solve a variant of image restoration problems without task-specific retraining, as long as the prior of the diffusion model is still suitable for the new testing image requiring restoration. The training-free benefit leads to a raised interest of new inverse problem solving approach via diffusion model. 

\begin{remark}[ODE modeling for diffusion model]
  Except for the SDE modeling of diffusion model, another probability flow-based ODE modeling is also widely used for generative models. Our developed scheme for diffusion posterior sampling actually contains the ODE modeling as a special case.
\end{remark}

\subsection{Tweedie's formula for denoising}

In the context of generation process of diffusion model, at each timestep $t$, the core is estimating the expectation $\mathbb{E}[\bm{x}_0|\bm{x}_t]$ of $p(\bm{x}_0|\bm{x}_t)$. Corresponding to the Gaussian corruption following $p(\bm{x}_t|\bm{x}_0)$ in the forward process, the expectation estimate $\mathbb{E}[\bm{x}_0|\bm{x}_t]$ actually implements the denoising of $\bm{x}_t$. Though the distribution $p(\bm{x}_0|\bm{x}_t)$ is intractable, its mean and variance can be determined through the Tweedie's formula.

\begin{lemma}[Mean and variance of $p(\bm{x}_0|\bm{x}_t)$~\cite{boys2023tweedie}]\label{lem:1}
  Given the Gaussian conditional density $p(\bm{x}_t|\bm{x}_0)\sim\mathcal{N}(\bm{x}_t|\alpha_t\bm{x}_0,\sigma_t^2\bm{I})$, the mean and variance of $p(\bm{x}_0|\bm{x}_t)$ are given by
  \begin{equation}\label{eq:twee}
    \begin{aligned}
      \mathbb{E}[\bm{x}_0|\bm{x}_t] &= \frac{\bm{x}_t+\sigma_t^2\nabla_{\bm{x}_t}\log p(\bm{x}_t)}{\alpha_t}\\
      \mathbb{V}[\bm{x}_0|\bm{x}_t] &= \frac{\sigma_t^2}{\alpha_t^2}\left(\bm{I}+\sigma_t^2\nabla_{\bm{x}_t}^2\log p(\bm{x}_t)\right).
    \end{aligned}
  \end{equation}
\end{lemma}
The Tweedie's formula, especially its expectation $\mathbb{E}[\bm{x}_0|\bm{x}_t]$ is leveraged to drive the numerical approximation to the reverse sampling SDE~\eqref{eq:uncond}. The score function $\nabla_{\bm{x}_t}\log p(\bm{x}_t)$ in~\eqref{eq:twee} is replaced by the learned diffusion model $s_{\theta}(\bm{x}_t,t)$. 

\subsection{Diffusion posterior sampling in discretization case}

There are three common configurations of the noise schedules for the parameters $\alpha_t$ and $\sigma_t$ for the forward noise corruption process~\cite{song2020score}. Out of the three configurations, we will leverage the pretrained model trained from two groups of configurations. For comprehensive presentation, we review the two experimental configurations: VE SDE~\cite{song2019generative} and VP SDE~\cite{ho2020denoising} respectively. From continuous perspective, VE SDE sets 
$f(t)=0, g(t)=\sqrt{\frac{d[\sigma^2(t)]}{dt}}$, while VP SDE sets $f(t)=-\frac{1}{2}\beta(t)$, $g(t)=\sqrt{\beta(t)}$. For both of them, the transition distribution~\eqref{eq:trans} can be computed. For VE SDE, $\alpha_t =\alpha_0= 1, \sigma_t^2 = \sigma^2(t)$, For VP SDE, $\alpha_t=\exp(-\frac{1}{2}\int_0^t\beta(s) ds), \sigma_t^2 = \int_0^t \exp(\int_{\tau}^t-\beta(s)ds)\beta(\tau)d\tau$. 
For VP SDE, we set $\beta(t)=\bar{\beta}_{\min} + t(\bar{\beta}_{\max}-\bar{\beta}_{\min})$ for $t\in [0,1]$. Following Ho \etal~\cite{ho2020denoising}, we adopt the following discretization scheme. We define $\beta_i:=\beta(i/N)$ for $i=1,\ldots,N$ and  define $\overline{\alpha}_i := \Pi_{j=1}^i (1-\beta_j)$. We denote $\bm{x}_i:=\bm{x}_{t_i}$ for simplicity. Hence the discretized transition distribution $p(\bm{x}_i|\bm{x}_0)$ satisfies $\alpha_t = \sqrt{\overline{\alpha}_i}$ and $\sigma_t^2 = 1-\overline{\alpha}_i$. For VE SDE, we set $\sigma(t)= \sigma_{\min}\left(\frac{\sigma_{\max}}{\sigma_{\min}}\right)^t$ and define $\sigma_t:=\sigma(i/N)$ for $i=1,\ldots,N$. For the discretized transition, we have $p(\bm{x}_i|\bm{x}_0)=\mathcal{N}(\bm{x}_0,\sigma_i^2\bm{I})$. Hereafter, for simple notations, we use the index $t$ and $i$ interchangeably, 

With the conditional reverse SDE for posterior sampling~\eqref{eq:cond_sde}, in discretization case, the general method can be achieved using the two-step iteration, which comprises of two consecutive steps
\begin{equation}
  \label{eq:con}
  \begin{aligned}
    \bm{x}_{t-1} &= \text{Unconditional sampling}(\bm{x}_t,s_{\theta}(\bm{x}_t,t))\\
    \bm{x}_{t-1} &=\bm{x}_{t-1} -\xi_t\nabla_{\bm{x}_t}\log p(\bm{y}|\bm{x}_t),
  \end{aligned}
\end{equation}
where the unconditional sampling can be selected form DDPM~\cite{ho2020denoising} and DDIM~\cite{song2021denoising} and $\xi_t$ is the stepsize (strength level) for the measurement guidance. The remaining thing is to estimate the intractable term $\nabla_{\bm{x}_t}\log p(\bm{y}|\bm{x}_t)$. There are a line of works aiming to this target by approximating the noisy likelihood $p(\bm{y}|\bm{x}_t)$ first and its gradient is achieved by direct computation.

\subsubsection{Existing approximations of the likelihood and their issues}\label{sec:ex}

As only the likelihood function $p(\bm{y}|\bm{x}_0)$ is known, one can factorize the $p(\bm{y}|\bm{x}_t)=\int p(\bm{y}|\bm{x}_0)p(\bm{x}_0|\bm{x}_t) d\bm{x}_0$.
DPS proposed by~\cite{chung2022diffusion} leveraged the delta-like approximation $p(\bm{y}|\bm{x}_t)\simeq p(\bm{y}|\hat{\bm{x}}_0(\bm{x}_t))$, where $\hat{\bm{x}}_0(\bm{x}_t)$ is the expectation $\mathbb{E}[\bm{x}_0|\bm{x}_t]$ using Tweedie's formula with learned score network:
\begin{equation}\label{eq:uncs}
  \hat{\bm{x}}_0(\bm{x}_t) = \frac{\bm{x}_t+\sigma_t^2s_{\theta}(\bm{x}_t,t)}{\alpha_t}.
\end{equation}
For Gaussian noisy linear problems, DPS approximated
\begin{equation}
  p(\bm{y}|\bm{x}_t) \sim \mathcal{N}(\bm{A}\hat{\bm{x}}_0(\bm{x}_t),\sigma_y^2\bm{I}).
\end{equation}
Instead of using the approximation $p(\bm{x}_0|\bm{x}_t)=\delta(\bm{x}_0-\hat{\bm{x}}_0(\bm{x}_t))$, $\Pi$GDM proposed by~\cite{song2022pseudoinverse} leveraged Gaussian approximation $p(\bm{x}_0|\bm{x}_t)\sim \mathcal{N}(\hat{\bm{x}}_0(\bm{x}_t),r_t^2\bm{I})$ ($r_t^2$'s are preconfigured hyperparameters).
For Gaussian noisy linear problems, $\Pi$GDM approximated
\begin{equation}
  p(\bm{y}|\bm{x}_t) \sim \mathcal{N}(\bm{A}\hat{\bm{x}}_0(\bm{x}_t), \sigma_y^2\bm{I}+r_t^2\bm{A}\bm{A}^T).
\end{equation}

When computing the gradient $\nabla_{\bm{x}_t}\log p(\bm{y}|\bm{x}_t)$, both of DPS and $\Pi$GDM require the backpropagation through the pretrained diffusion model. As a result, they are computationally demanding and memory intensive. To address such issue, for VE SDE, \cite{jalal2021robust} assumed that likelihood $p(\bm{y}|\bm{x}_t)$ is a Gaussian $\mathcal{N}(\bm{y}|\bm{A}\bm{x}_t, (\sigma_y^2+\gamma_t^2)\bm{I})$ with $\gamma_t^2$ being the hyperparameters such that, with $t$ decreasing, the assumed noise level is decreasing towards 0. For VP SDE, \cite{meng2022diffusion} adopted free-of-pretrained-model approximation $p(\bm{x}_0|\bm{x}_t)\sim\mathcal{N}(\frac{1}{\alpha_t}\bm{x}_t,\frac{\sigma_t^2}{\alpha_t^2}\bm{I})$ with a crude uninformative prior of $p(\bm{x}_0)$, which leads to the approximation 
\begin{equation}\label{eq:dmps}
  p(\bm{y}|\bm{x}_t) \sim \mathcal{N}(\frac{1}{\alpha_t}\bm{A}\bm{x}_t, \sigma_y^2\bm{I}+\frac{\sigma_t^2}{\alpha_t^2}\bm{A}\bm{A}^T).
\end{equation}
The approximation~\eqref{eq:dmps} is very heuristic and its performance degrades much for highly ill-posed problems.

\section{Methodology}\label{sec:meth}

Unlike the previous approaches of diffusion posterior sampling for inverse problem, instead of implementing the \emph{post-guidance} following the unconditional sampling~\eqref{eq:con}, we consider modifying the intermediate step of unconditional sampling like the work~\cite{wang2022zero}. The existing methods, such as~\cite{chung2022diffusion,song2022pseudoinverse}, are built on standard DDPM or DDIM unconditional samplers, which loss the flexibility of controlling the injected noise level. To increase flexibility, we first introduce a general parameterized DDIM-like unconditional sampling scheme involving the expectation $\mathbb{E}[\bm{x}_0|\bm{x}_t]$, which contains DDPM and DDIM as two special cases. The transition from the unconditional sampling to conditional sampling means to replacing expectation $\mathbb{E}[\bm{x}_0|\bm{x}_t]$ by measurement-guided expectation $\mathbb{E}[\bm{x}_0|\bm{x}_t,\bm{y}]$ of distribution $p(\bm{x}_0|\bm{x}_t,\bm{y})$. With this perspective, we propose using the maximum-a-posteriori (MAP) surrogate to $\mathbb{E}[\bm{x}_0|\bm{x}_t,\bm{y}]$ instead of other subtle approximation to alleviate the computation issue of the existing DPS-like methods. As a comparison to~\cite{wang2022zero}, our method enables more friendly SVD-free implementation for noisy measurement.

\subsection{Conditional sampling driven by MAP approximation}\label{sec:main}

In the following, built on a pretrained diffusion model, we consider establishing the diffusion posterior sampling based on the more general parameterized DDIM-like sampling method. 

Similar to the non-Markov chain based sampling proposed in~\cite{song2021denoising}, we propose the following parameterized unconditional sampling scheme
\begin{equation}
  \label{eq:unc}
  \bm{x}_{t-1} = \alpha_{t-1}\mathbb{E}[\bm{x}_0|\bm{x}_t] + \sigma_{t-1}(\sqrt{1-\xi}\frac{\bm{x}_t-\alpha_t\mathbb{E}[\bm{x}_0|\bm{x}_t]}{\sigma_t}+\sqrt{\xi}\bm{z}),
\end{equation}
where $\mathbb{E}[\bm{x}_0|\bm{x}_t]$ is the expectation of density $p(\bm{x}_0|\bm{x}_t)$ and $0\leq\xi\leq 1$ is a hyperparameter. The role of hyperparameter $\xi$ serves as the purification~\cite{li2024decoupled} for ill-posed problem. Another function of the extra noise $\bm{z}$ helps improve the restoration performance of inverse problems, inspired from~\cite{xu2023restart}.

To transform the unconditional sampling~\eqref{eq:unc} to measurement-guided conditional sampling, we consider replacing the expectation estimation $\mathbb{E}[\bm{x}_0|\bm{x}_t]$ by measurement-aware expectation $\mathbb{E}[\bm{x}_0|\bm{x}_t,\bm{y}]$. The resulting generation scheme is
\begin{equation}
  \label{eq:cond}
  \bm{x}_{t-1} = \alpha_{t-1}\mathbb{E}[\bm{x}_0|\bm{x}_t,\bm{y}] + \sigma_{t-1}(\sqrt{1-\xi}\frac{\bm{x}_t-\alpha_t\mathbb{E}[\bm{x}_0|\bm{x}_t,\bm{y}]}{\sigma_t}+\sqrt{\xi}\bm{z}).
\end{equation}

Equation~\eqref{eq:uncs} provided the common estimate $\hat{\bm{x}}_0(\bm{x}_t)\simeq \mathbb{E}[\bm{x}_0|\bm{x}_t]$. With the estimate, scheme~\eqref{eq:unc} resembles the DDIM and DDPM as special cases.

To develop a practical conditional generation scheme of~\eqref{eq:cond}, the remaining thing is to estimate $\mathbb{E}[\bm{x}_0|\bm{x}_t,\bm{y}]$. Suppose the joint distribution between $\bm{x}_0,\bm{x}_t,\bm{y}$ is given by $p(\bm{x}_0,\bm{y},\bm{x}_t)= p(\bm{x}_0)p(\bm{y}|\bm{x}_0)p(\bm{x}_t|\bm{x}_0)$. Hence, the mean of $p(\bm{x}_0|\bm{x}_t,\bm{y})$ can be computed using the following result.

\begin{theorem}[~\cite{peng2024imp}\label{thm:1}]
  Suppose the joint distribution between $\bm{x}_0,\bm{x}_t,\bm{y}$ is given by
  \begin{equation}
    p(\bm{x}_0,\bm{y},\bm{x}_t)= p(\bm{x}_0)p(\bm{y}|\bm{x}_0)p(\bm{x}_t|\bm{x}_0).
  \end{equation}
  Hence, the mean of $p(\bm{x}_0|\bm{x}_t,\bm{y})$
\begin{equation}\label{eq:res}
  \mathbb{E}[\bm{x}_0|\bm{x}_t,\bm{y}] = \mathbb{E}[\bm{x}_0|\bm{x}_t] + \frac{\sigma_t^2}{\alpha_t}\nabla_{\bm{x}_t}\log p(\bm{y}|\bm{x}_t).
\end{equation}
\end{theorem}

The conditional generation scheme~\eqref{eq:cond} can be reformulated as two consecutive iterations 
\begin{equation}\label{eq:eq5}
  \begin{aligned}
    \bm{x}_{t-1} &= \alpha_{t-1}\mathbb{E}[\bm{x}_0|\bm{x}_t] + \sigma_{t-1}(\sqrt{1-\xi}\frac{\bm{x}_t-\alpha_t\mathbb{E}[\bm{x}_0|\bm{x}_t]}{\sigma_t}+\sqrt{\xi}\bm{z})\\
    \bm{x}_{t-1} &=\bm{x}_{t-1} +  \left(\alpha_{t-1}-\frac{\sigma_{t-1}\alpha_t\sqrt{1-\xi}}{\sigma_t}\right)\cdot (\mathbb{E}[\bm{x}_0|\bm{x}_t,\bm{y}]-\mathbb{E}[\bm{x}_0|\bm{x}_t]).
  \end{aligned}
\end{equation}

Using the equality~\eqref{eq:res}, the measurement guidance step of~\eqref{eq:eq5} requires the estimation of $\nabla_{\bm{x}_t}\log p(\bm{y}|\bm{x}_t)$, as $\mathbb{E}[\bm{x}_0|\bm{x}_t,\bm{y}]-\mathbb{E}[\bm{x}_0|\bm{x}_t]=\frac{\sigma_t^2}{\alpha_t}\nabla_{\bm{x}_t}\log p(\bm{y}|\bm{x}_t)$. If we follow the estimations from~\cite{chung2022diffusion,song2022pseudoinverse}, the solving method is still computationally demanding. To address the issue, we propose a less precise surrogate estimate to $\mathbb{E}[\bm{x}_0|\bm{x}_t,\bm{y}]$ in~\eqref{eq:eq5}.

Instead of using the right hand side of~\eqref{eq:res} to predict the expectation $\mathbb{E}[\bm{x}_0|\bm{x}_t,\bm{y}]$ of $p(\bm{x}_0|\bm{x}_t,\bm{y})$, we consider the maximum likelihood estimation of $p(\bm{x}_0|\bm{x}_t,\bm{y})$ as a proxy of $\mathbb{E}[\bm{x}_0|\bm{x}_t,\bm{y}]$. In other words, we take the maximum mode of $p(\bm{x}_0|\bm{x}_t,\bm{y})$ as an approximation to the expectation. Though there exists approximation error between the mean and the maximum mode, such accumulation error can be rectified by the randomness of $\bm{z}$ for the per iteration~\cite{xu2023restart}, where the authors pointed out that the approximation error of the score function will dominate when the diffusion step is large and the stochastic nature of the general sampling scheme helps forget the accumulated errors from the previous diffusion steps. Compared to deterministic sampling scheme, the injected pure Gaussian noise of general sampling scheme helps the accumulated error from the score function error undergoes contraction by a constant factor. For the conditional generation, the accumulated error from the conditional score function error shall be contracted by the injected noise per iteration as well.

Using Bayes' rule, we have
\begin{equation}
  p(\bm{x}_0|\bm{x}_t,\bm{y})\propto \frac{p(\bm{y}|\bm{x}_0)p(\bm{x}_0|\bm{x}_t)p(\bm{x}_t)}{p(\bm{x}_t,\bm{y})}.
\end{equation}
The distribution $p(\bm{x}_0|\bm{x}_t)$ is intractable, we also use a Gaussian distribution to approximate it. From the moment estimation~\eqref{eq:twee} in the Tweedie's formula, we suppose that $p(\bm{x}_0|\bm{x}_t)\sim\mathcal{N}(\bm{x}_0|\hat{\bm{x}}_0(\bm{x}_t), \frac{\sigma_t^2}{\alpha_t^2}\bm{I})$. Therefore, the maximum likelihood leads to the optimization problem:
\begin{equation}
  \label{eq:ma}
  \min_{\bm{x}}\quad \frac{1}{2}\norm{\bm{y}-\bm{A}\bm{x}}_2^2 + \frac{\sigma_y^2}{2\sigma_t^2/\alpha_t^2}\norm{\bm{x}-\hat{\bm{x}}_0(\bm{x}_t)}_2^2.
\end{equation}
In this way, we are required to solve~\eqref{eq:ma} as the core of conditional sampling method. The optimization~\eqref{eq:ma} resembles the traditional regularized approach for inverse problem, and the prior is diffusion model based. This optimization problem~\eqref{eq:ma} is different from the one in~\cite{song2021solving}, where there are two corruption paths considered.

For linear problem, we have the closed-form solution to the above optimization
\begin{equation}
  \bm{x}:=(\sigma_y^2\bm{I}+\frac{\sigma_t^2}{\alpha_t^2}\bm{A}^T\bm{A})^{-1}(\sigma_y^2\hat{\bm{x}}_0(\bm{x}_t)+\frac{\sigma_t^2}{\alpha_t^2}\bm{A}^T\bm{y}).
\end{equation}

Hence using the estimator $\bm{x}\simeq\text{MAP}[\bm{x}_0|\bm{x}_t,\bm{y}]\simeq \mathbb{E}[\bm{x}_0|\bm{x}_t,\bm{y}]$ and the Tweedie's formula for $\mathbb{E}[\bm{x}_0|\bm{x}_t])$, we have the estimation
\begin{equation}
    \mathbb{E}[\bm{x}_0|\bm{x}_t,\bm{y}]-\mathbb{E}[\bm{x}_0|\bm{x}_t]\simeq \frac{\sigma_t^2}{\alpha_t^2}(\sigma_y^2\bm{I}+\frac{\sigma_t^2}{\alpha_t^2}\bm{A}^T\bm{A})^{-1}\bm{A}^T(\bm{y}-\bm{A}\hat{\bm{x}}_0(\bm{x}_t)).
  \end{equation}
  Hence the \emph{post measurement guidance} step is
  \begin{equation}
  \label{eq:ours}
  \bm{x}_{t-1} = \bm{x}_{t-1} + \textcolor{blue}{\alpha_{t-1}\left(\frac{\sigma_t}{\alpha_t}-\sqrt{1-\xi}\frac{\sigma_{t-1}}{\alpha_{t-1}}\right)\sigma_t}\left[\frac{1}{\alpha_t}(\sigma_y^2\bm{I}+\frac{\sigma_t^2}{\alpha_t^2}\bm{A}^T\bm{A})^{-1}\bm{A}^T(\bm{y}-\bm{A}\hat{\bm{x}}_0(\bm{x}_t))\right].
\end{equation}

Using $\mathbb{E}[\bm{x}_0|\bm{x}_t,\bm{y}]-\mathbb{E}[\bm{x}_0|\bm{x}_t]=\frac{\sigma_t^2}{\alpha_t}\nabla_{\bm{x}_t}\log p(\bm{y}|\bm{x}_t)$, our MAP-based expectation proxy actually derives a novel approximation
\begin{equation}
  \nabla_{x_t}\log p(\bm{y}|\bm{x}_t)\simeq \frac{1}{\alpha_t}(\sigma_y^2\bm{I}+\frac{\sigma_t^2}{\alpha_t^2}\bm{A}^T\bm{A})^{-1}\bm{A}^T(\bm{y}-\bm{A}\hat{\bm{x}}_0(\bm{x}_t)).
\end{equation}

Note that the matrix $\bm{B}:=(\alpha_t^2\sigma_y^2\bm{I}+\sigma_t^2\bm{A}^T\bm{A})$ is positive semi-definite. To efficiently compute the gradient term, we can leverage the conjugate gradient (CG) to obtain the solution. Unlike DDRM~\cite{kawar2022denoising}, which depends on the SVD of $\bm{A}$, our method does not require the SVD. It widens the applicability for a general inverse problem. 

If the singular-value decomposition $\bm{A}=\bm{U}\bm{\Sigma}\bm{V}^T$ is available, where $\bm{U}\in\mathbb{R}^{m\times m},\bm{\Sigma}\in\mathbb{R}^{m\times n}$ and $\bm{V}\in\mathbb{R}^{n\times n}$, the matrix inverse $\bm{B}^{-1}$ has the exact expression:
\begin{equation}
  \bm{B}^{-1} = \bm{V}\diag\left(\frac{1}{\alpha_t^2\sigma_y^2 + \sigma_t^2\bm{s}^2}\right)\bm{V}^T,
\end{equation}
where $\bm{s}=[s_1,\ldots,s_m]$ and $s_i$'s are the singular values.

\subsection{Connection to existing approaches}

\paragraph{Comparison to $\Pi$GDM} Following the $\Pi$GDM approximation and the hyperparameter configuration in~\cite{meng2022diffusion}, we know that
\begin{equation}
  p(\bm{y}|\bm{x}_t) \sim \mathcal{N}(\bm{A}\hat{\bm{x}}_0(\bm{x}_t), \sigma_y^2\bm{I}+\frac{\sigma_t^2}{\alpha_t^2}\bm{A}\bm{A}^T).
\end{equation}
Hence, we can compute the gradient
\begin{equation}
  \label{eq:full}
  \begin{aligned}
    \nabla_{\bm{x}_t}^{\text{$\Pi$GDM}}\log p(\bm{y}|\bm{x}_t) & \simeq \left(\frac{\partial \hat{\bm{x}}_0(\bm{x}_t)}{\partial \bm{x}_t}\right)^T\bm{A}^T(\sigma_y^2\bm{I}+\frac{\sigma_t^2}{\alpha_t^2}\bm{A}\bm{A}^T)^{-1}(\bm{A}\hat{\bm{x}}_0(\bm{x}_t)-\bm{y})\\
    & = \left(\frac{1}{\alpha_t}+\frac{\sigma_t^2}{\alpha_t}\frac{\partial \bm{s}_{\theta}(\bm{x}_t,t)}{\partial \bm{x}_t}\right)^T\bm{A}^T(\sigma_y^2\bm{I}+\frac{\sigma_t^2}{\alpha_t^2}\bm{A}\bm{A}^T)^{-1}(\bm{A}\hat{\bm{x}}_0(\bm{x}_t)-\bm{y}).
  \end{aligned}
\end{equation}
As a comparison to~\eqref{eq:ours}, $\Pi$GDM requires computation of the backpropagation through the diffusion model, which is memory intensive and computationally demanding. It is suggested that the term $\frac{\partial \hat{\bm{x}}_0(\bm{x}_t)}{\partial \bm{x}_i}$ leads to manifold preservation~\cite{chung2022improving}, which is helpful for the image restoration. In our experiment, the extra term is not necessary for image restoration. And the derivation may lead to unstable performance.

\paragraph{Comparison to Meng's DMPS approach} The update scheme of our method is very similar to the work of Meng and Kabashima~\cite{meng2022diffusion}. In the following, we compare our approach to the work. The measurement guidance gradient estimation of Meng and Kabashima is
\begin{equation}
  \label{eq:meng}
   \nabla_{\bm{x}_t}^{\text{Meng}}\log p(\bm{y}|\bm{x}_t) =\frac{1}{\alpha_t}\bm{A}^T(\sigma_y^2\bm{I}+\frac{\sigma_t^2}{\alpha_t^2}\bm{A}\bm{A}^T)^{-1}(\bm{A}(\frac{\bm{x}_t}{\alpha_t})-\bm{y}).
\end{equation}
It is highly related to our scheme~\eqref{eq:ours}. Comparing eqn~\eqref{eq:meng} and eqn~\eqref{eq:ours}, the only difference is the evaluation point to perform the gradient calculation. Except for the evaluated point, the equivalence can be achieved using the following lemma.

\begin{lemma}
  If $\bm{A}\in\mathbb{R}^{m\times n}$, then the following equality
  \begin{equation}
    (\alpha_t^2\sigma_y^2\bm{I}+\sigma_t^2\bm{A}^T\bm{A})^{-1}\bm{A}^T = \bm{A}^T(\alpha_t^2\sigma_y^2\bm{I}+\sigma_t^2\bm{A}\bm{A}^T)^{-1}
  \end{equation}
  holds.
\end{lemma}
\begin{proof}
  Using the SVD of $\bm{A}$, it is straightforward to check the equality of the two matrices.
\end{proof}

\paragraph{Comparison to DDRM~\cite{kawar2022denoising} and DDNM~\cite{wang2022zero}}
Both of DDNM and DDRM implemented the algorithms using the SVD of the linear operator. Unlike our \emph{unified} expression to deal with noisy inverse problems, for noisy cases, DDRM and DDNM$+$ are eigenspace-specific and \emph{intricate}, as the iteration varies for each eigenvector. For noiseless cases, \ie, $\sigma_y=0$, the implementations of DDRM and DDNM degenerate to a compact expression. %

With the same philosophy, when $\sigma_y=0$, the maximum likelihood estimation~\eqref{eq:ma} solves
\begin{equation}
  \min_{\bm{x}}\quad \frac{1}{2}\norm{\bm{x}-\hat{\bm{x}}_0(\bm{x}_t)}_2^2,\quad \text{s.t. } \bm{y}=\bm{A}\bm{x}. 
\end{equation}
Hence the estimate $\bm{x}$ of $\mathbb{E}[\bm{x}_0|\bm{x}_t,\bm{y}]$ is the projection of $\hat{\bm{x}}_0(\bm{x}_t)$ onto the set $\{\bm{x}\mid \bm{y}=\bm{A}\bm{x}\}$. The solution is
\begin{equation}
  \bm{x} = \hat{\bm{x}}_0(\bm{x}_t) + \bm{A}^{\dagger}(\bm{y}-\bm{A}\hat{\bm{x}}_0(\bm{x}_t)). 
\end{equation}
Hence we have the measurement guidance iteration
\begin{equation}
  \label{eq:ddnm}
  \bm{x}_{t-1} = \bm{x}_{t-1} +\left(\alpha_{t-1}-\frac{\sigma_{t-1}\alpha_t\sqrt{1-\xi}}{\sigma_t}\right)\bm{A}^{\dagger}(\bm{y}-\bm{A}\hat{\bm{x}}_0(\bm{x}_t)).
\end{equation}
As a comparison to~\eqref{eq:ours}, the sampling method~\eqref{eq:ddnm} is a special case of~\eqref{eq:ours} with $\sigma_y=0$.

With configuration $\eta_b=1$, the spectral inpainting based DDRM iteration~\cite{kawar2022denoising} is
\begin{equation*}
  \bm{x}_{t-1} = \bm{x}_{t-1} + \bm{A}^{\dagger}(\alpha_t\bm{y}\textcolor{blue}{+\sigma_t\bm{z}}-\bm{A}\bm{x}_t) = \bm{x}_{t-1} + \alpha_t\bm{A}^{\dagger}\left(\bm{y}-\bm{A}\hat{\bm{x}}_0(\bm{x}_t)+\frac{\sigma_t}{\alpha_t}\bm{z}+\bm{A}(\frac{\sigma_t^2}{\alpha_t}\bm{s}_{\theta}(\bm{x}_t,t))\right).
\end{equation*}
Compared to~\eqref{eq:ours}, there are two differences: the extra noise term $\frac{\sigma_t}{\alpha_t}\bm{z}+\bm{A}(\frac{\sigma_t^2}{\alpha_t}\bm{s}_{\theta}(\bm{x}_t,t))$ and the stepsize.

\subsection{Extension to special non-linear case}

Considering an extension, we leverage specific approximation to generalize the formula~\eqref{eq:ours} for some specific non-linear problems, similar to the work~\cite{kawar2022jpeg}.

For the specific case $\bm{A}\bm{A}^T=\bm{I}$, equation~\eqref{eq:ours} leads to
\begin{equation}
  \label{eq:ls}
  \bm{x}_{t-1} = \bm{x}_{t-1} + \textcolor{blue}{\alpha_{t-1}\left(\frac{\sigma_t}{\alpha_t}-\sqrt{1-\xi}\frac{\sigma_{t-1}}{\alpha_{t-1}}\right)\sigma_t}\frac{1}{\sigma_y^2+\sigma_t^2/\alpha_t^2}\left[\frac{1}{\alpha_t}(\bm{A}^T(\bm{y}-\bm{A}\hat{\bm{x}}_0(\bm{x}_t)))\right].
\end{equation}
 Analogously, for some non-linear measurement function $h:\mathbb{R}^n\to\mathbb{R}^m$, we may find another function $h^T:\mathbb{R}^m\to\mathbb{R}^n$ such that $h^T(h(\bm{x}))\simeq \bm{x}$ for all $\bm{x}\in\mathbb{R}^n$. Two non-linear examples, including the quantization and JPEG compression, exhibit such properties, as discussed in~\cite{kawar2022jpeg}. For both of them, let $h(\bm{x})$ be the degradation operator, and we define $h^T(\bm{x})=\bm{x}$, eqn.~\eqref{eq:ls} can solve such specific non-linear problems.

\subsection{Theoretical analysis of our estimation}

The error analysis of our MAP approximation to $\mathbb{E}[\bm{x}_0|\bm{x}_t,\bm{y}]$ for an intractable $p(\bm{x}_0|\bm{x}_t,\bm{y})$ is impossible. Assume $p(\bm{y}|\bm{x}_0)$ is Gaussian, then our MAP-based method can be recognized as another approximation to $\nabla_{x_t}\log p(\bm{y}|\bm{x}_t)$ as we have presented. And our MAP-based estimate is highly related to the approximation $\nabla_{x_t}\log p^{\text{app}}(\bm{y}|\bm{x}_t)$, where $p^{\text{app}}(\bm{y}|\bm{x}_t)=\int p(\bm{y}|\bm{x}_0)p^{\text{app}}(\bm{x}_0|\bm{x}_t)d\bm{x}_0$ and $p^{\text{app}}(\bm{x}_0|\bm{x}_t)$ is an estimate of $p(\bm{x}_0|\bm{x}_t)$. Assume when the distance $|p^{\text{app}}(\bm{y}|\bm{x}_t)-p(\bm{y}|\bm{x}_t)|\leq \epsilon$ uniformly for $\bm{x}_t\in\mathbb{R}^d$,~\footnote{This assumption is not trivial, for a bounded uniformly continuous and everywhere differentiable function, its gradient function is unnecessarily bounded.} there exists a constant $M>0$, such that
\begin{equation*}
  \norm{\nabla_{x_t}\log p^{\text{app}}(\bm{y}|\bm{x}_t)-\nabla_{x_t}\log p(\bm{y}|\bm{x}_t)}_2\leq M\epsilon.
\end{equation*}
The remaining thing is to characterize the distance $|p^{\text{app}}(\bm{y}|\bm{x}_t)-p(\bm{y}|\bm{x}_t)|$.

\begin{theorem}\label{thm:2}
  Let $p^{\text{mid}}(\bm{x}_0|\bm{x}_t)=\mathcal{N}(\bm{x}_0;\mathbb{E}[\bm{x}_0|\bm{x}_t],\mathbb{V}[\bm{x}_0|\bm{x}_t])$ be the mean-field approximation to the intractable $p(\bm{x}_0|\bm{x}_t)$. Suppose the total variation distance $\text{TV}(p^{\text{mid}}(\bm{x}_0|\bm{x}_t),p(\bm{x}_0|\bm{x}_t))\leq \epsilon_{\text{MF}}$. For a given approximation density $p^{\text{app}}(\bm{x}_0|\bm{x}_t)=\mathcal{N}(\bm{x}_0;\mu,\Sigma)$ to $p^{\text{mid}}(\bm{x}_0|\bm{x}_t)$, suppose that $\text{TV}(p^{\text{app}}(\bm{x}_0|\bm{x}_t),p^{\text{mid}}(\bm{x}_0|\bm{x}_t))\leq \frac{1}{600}$, then we have
  \begin{equation}
    \text{TV}(p^{\text{app}},p^{\text{mid}})\leq \frac{1}{\sqrt{2}}\left(\norm{\Sigma^{-1/2}\mathbb{V}[\bm{x}_0|\bm{x}_t]\Sigma^{-1/2}-I_d}_F + \norm{\Sigma^{-1/2}(\mu-\mathbb{E}[\bm{x}_0|\bm{x}_t])}_2\right).
  \end{equation}
  Furthermore, we have
  \begin{equation}
    |p^{\text{app}}-p|\leq 2(\epsilon_{\text{MF}}+\frac{1}{\sqrt{2}}\norm{\Sigma^{-1/2}\mathbb{V}[\bm{x}_0|\bm{x}_t]\Sigma^{-1/2}-I_d}_F + \norm{\Sigma^{-1/2}(\mu-\mathbb{E}[\bm{x}_0|\bm{x}_t])}_2)\cdot \int p(\bm{y}|\bm{x}_0) d\bm{x}_0.
  \end{equation}
\end{theorem}

Before we present the proof, we cite the known result of the upper bound of the TV distance of two Gaussians.
\begin{lemma}[Upper bound of TV distance of two multivariate Gaussians~\cite{arbas2023polynomial}]\label{lem:2}
   Let $\mu_1,\mu_2\in\mathbb{R}^d$ and $\Sigma_1,\Sigma_2\in\mathbb{R}^{d\times d}$ be positive-definite matrices. Suppose $\text{TV}(\mathcal{N}(\mu_1,\Sigma_1),\mathcal{N}(\mu_2,\Sigma_2))<\frac{1}{600}$. Let
   \begin{equation}
     \Delta = \norm{\Sigma_1^{-1/2}\Sigma_2\Sigma_1^{-1/2}-I_d}_F + \norm{\Sigma_1^{-1/2}(\mu_1-\mu_2)}_2.
   \end{equation}
   Then
   \begin{equation}
     \text{TV}(\mathcal{N}(\mu_1,\Sigma_1),\mathcal{N}(\mu_2,\Sigma_2))\leq \frac{1}{\sqrt{2}}\Delta.
   \end{equation}
 \end{lemma}

  \begin{proof}[Proof of Theorem~\ref{thm:2}]

Though $p(\bm{y}|\bm{x}_t)$ is intractable, we have its exact mean and variance characterization in Lemma~\ref{lem:1}. Using the decomposition
 \begin{equation}\label{eq:decomp}
   p(\bm{y}|\bm{x}_t) = \int p(\bm{y}|\bm{x}_0)p(\bm{x}_0|\bm{x}_t) d\bm{x}_0,
 \end{equation}
for the given approximation $p^{\text{app}}(\bm{y}|\bm{x}_t)=\int p(\bm{y}|\bm{x}_0)p^{\text{app}}(\bm{x}_0|\bm{x}_t) d\bm{x}_0$ to $p(\bm{y}|\bm{x}_t)$, 
we have
 \begin{equation*}
   \begin{aligned}
     |p(\bm{y}|\bm{x}_t)&-p^{\text{app}}(\bm{y}|\bm{x}_t)|\\
     &\leq \int p(\bm{y}|\bm{x}_0)|p(\bm{x}_0|\bm{x}_t)-p^{\text{app}}(\bm{x}_0|\bm{x}_t)| d\bm{x}_0\\
     &\leq \int p(\bm{y}|\bm{x}_0)|p(\bm{x}_0|\bm{x}_t)-p^{\text{mid}}(\bm{x}_0|\bm{x}_t)| d\bm{x}_0 + \int p(\bm{y}|\bm{x}_0)|p^{\text{mid}}(\bm{x}_0|\bm{x}_t)-p^{\text{app}}(\bm{x}_0|\bm{x}_t)| d\bm{x}_0.\\
     &\leq 2(\text{TV}(p^{\text{mid}}(\bm{x}_0|\bm{x}_t),p(\bm{x}_0|\bm{x}_t)) + \text{TV}(p^{\text{mid}}(\bm{x}_0|\bm{x}_t),p^{\text{app}}(\bm{x}_0|\bm{x}_t)))\cdot \int p(\bm{y}|\bm{x}_0)d\bm{x}_0.
   \end{aligned}
 \end{equation*}
 Using Lemma~\ref{lem:2}, we complete the proof.
\end{proof}

Our MAP-based method configures $p^{\text{app}}(\bm{x}_0|\bm{x}_t)=\mathcal{N}(\hat{\bm{x}}_0(\bm{x}_t), \frac{\sigma_t^2}{\alpha_t^2})$ using Tweedie's formula. We omit the Jacobian matrix of score function $\nabla_{\bm{x}_t}^2\log p(\bm{x}_t)$ in the approximation $p^{\text{app}}(\bm{x}_0|\bm{x}_t)$, we can not expect the difference $|p^{\text{app}}(\bm{y}|\bm{x}_t)-p(\bm{y}|\bm{x}_t)|$ vanishes when both the error of the score function  $\varepsilon_{\text {score }}$ and the error of Jacobian matrix of score function $\varepsilon_{\mathrm{Jacobi}}$ vanish. Though the omission seems unreasonable, our approach saves computational cost without much performance scarification. 

\section{Experiments}\label{sec:exp}

We conduct experiments on a variant of inverse problems, including the image restoration tasks, medical sparse-view CT, JPEG decompression~\cite{kawar2022jpeg} to evaluate the performance of the proposed MAP-based conditional sampling method. For linear problems, we test our method on two cases: SVD-dependent and SVD-free linear operators. The JPEG decompression problem~\cite{kawar2022jpeg} is evaluated to show the extension of a specific nonlinear application with nondifferentiable degradation model.

For image restoration, two datasets, including Flickr Faces High Quality (FFHQ) $256\times 256$~\cite{karras2019style} and ImageNet $256\times 256$~\cite{deng2009imagenet} are considered. The publicly accessible pretrained score-based models for the two datasets are built on VP SDE. Hence in this case, the parameters in~\eqref{eq:unc} are $\alpha_t = \sqrt{\overline{\alpha}_t}, \sigma_t = \sqrt{1-\overline{\alpha}_t}$. The models are leveraged without finetuning for specific tasks. For FFHQ, we use the pretrained DDPM model from~\cite{choi2021ilvr}, and for ImageNet, we use the pretrained DDPM model from~\cite{dhariwal2021diffusion}. The performance of different posterior sampling methods is validated on 1000 images from both FFHQ and ImageNet datasets as the previous works does. The testing images for FFHQ are the first 1000 images while the 1000 images for ImageNet are selected as in~\cite{pan2020exploiting}.~\footnote{The selected 1K images can be referred in \url{https://github.com/XingangPan/deep-generatiVEprior/blob/master/scripts/imagenet_val_1k.txt}.}
For these datasets, we evaluate 100-diffusion-step iteration~\eqref{eq:ours}. To quantitatively compare the performance, we report the PSNR$\uparrow$, SSIM$\uparrow$ metric to measure the faithfulness to the original image, and the LPIPS$\downarrow$ metric to measure the perception quality of the restored image.

For CT reconstruction, we leveraged the diffusion model~\cite{song2021solving} based on VE SDE, which is trained on the 2016 American Association of Physicists in Medicine (AAPM) grand challenge dataset. In this case, the parameter in~\eqref{eq:unc} is set as $\alpha_t \equiv 1$. The estimate of $\mathbb{E}[\bm{x}_0|\bm{x}_t]$ is
\begin{equation}
  \hat{\bm{x}}_0(\bm{x}_t) = \bm{x}_t + \sigma_t^2\bm{s}_{\theta}(\bm{x}_t,t).
\end{equation}
For VE SDE, we follow the following unconditional sampling
\begin{equation}\label{eq:sm0}
  \bm{x}_{t-1} = \bm{x}_t + (\sigma_t^2-\sigma_{t-1}^2)\bm{s}_{\theta}(\bm{x}_t,t) + \sqrt{\sigma_t^2-\sigma_{t-1}^2}\bm{z}.
\end{equation}
Hence by the method in Section~\ref{sec:main}, the conditional measurement guided sampling method implements the following gradient descent after the unconditional iteration~\eqref{eq:sm0}
\begin{equation}\label{eq:vesde}
  \bm{x}_{t-1} = \bm{x}_{t-1} - (\sigma_t^2-\sigma_{t-1}^2)(\sigma_y^2\bm{I}+\sigma_t^2\bm{A}^T\bm{A})^{-1}\bm{A}^T(\bm{A}(\hat{\bm{x}}_0(\bm{x}_t))-\bm{y}).
\end{equation}
For sparse-view CT, we evaluate the above measurement-guided sampling with 1K sampling steps. For the CT performance evaluation, we provide only PSNR and SSIM metrics.

\subsection{Experiments for SVD-based linear problems}
\label{sec:svd-linear}

We first test and compare our MAP-based measurement guided diffusion sampling to the existing SVD-dependent methods for noisy linear problems. For a fair comparison, we adopted the degradation linear operators implemented in DDRM~\cite{kawar2022denoising}, where the SVD is provided. As a result, all the SVD-dependent methods can be included for comparison. We compare classical total variation regularized optimization method (denoted as L2TV), and diffusion model based methods, including DDRM~\cite{kawar2022denoising}, DDNM+~\cite{wang2022zero} for noisy problems, $\Pi$GDM~\cite{song2022pseudoinverse}, DMPS~\cite{meng2022diffusion} and the diffusion posterior sampling (DPS)~\cite{chung2022diffusion}. For L2TV, the data fidelity term is in squared $\ell_2$ metric and regularization parameter is set to $0.01$ for best performance cross all the problems. The results of other methods such as SNIPS~\cite{kawar2021snips}, Score-SDE~\cite{song2021solving}, and MCG~\cite{chung2022improving} are not shown as they performed inferior to DPS and DDRM in the mentioned image restoration tasks as demonstrated in~\cite{chung2022diffusion,kawar2022denoising}. For the compared methods, we use exactly the hyper-parameters as their default values. For $\Pi$GDM, we reimplemented it, as there is no available open-soured code. 

All images are normalized to the range $[0,1]$. Five tasks are considered as follows. (i) for deblurring with Gaussian kernel, the blur kernel is Gaussian with $\sigma=10$; (ii) for deblurring with a $9\times 9$ uniform kernel, and singular values below a certain threshold are zeroed, making the problem more ill-posed; (iii) for inpainting with random box mask, the mask of size $128\times 128$ is generated; (iv) and for inpainting, we mask parts of the original image with randomly $92\%$ of the pixels dropped;  (v) for super-resolution, we use a block averaging filter to downscale the images by a factor of 4 in each axis; Gaussian noise is added to the measurement domain with $\sigma_y=0.05$.

\paragraph{Effect of $\xi$ on restoration performance}
Note that there is a hyperparameter $\xi$ in our general diffusion posterior sampling scheme~\eqref{eq:cond}. Before proceeding, we conduct the ablation study of $\xi$ on the restoration performance. We selected the first 12 images from the 1k test images of FFHQ and ImageNet datasets and tested our method with different values of $\xi$ from the set $\{0.0,0.2,0.4,0.6,0.8,1.0\}$. The quantitative results of our method with different $\xi$ are listed in Tables~\ref{tab:1x} and~\ref{tab:2x}. See Figures~\ref{fig:ffhq_gauss_x} and~\ref{fig:img_gauss_x} for the visualization comparison. The best PSNR/SSIM performance is achieved when $\xi=1.0$, and LPIPS sometimes achieved best performance using $\xi=0.8$. For best PSNR/SSIM performance, our method configures $\xi=1$ hereafter. The extra injected pure Gaussian noise helps improve the restoration performance for inverse problems. The phenomenon aligned with the findings for unconditional generation in the works~\cite{xu2023restart,karras2022elucidating}.

\begin{table}[!htp]
  \centering 
  \caption{Quantitative results of our method with different $\xi$ on the FFHQ dataset. Inputs have an additive noise with $\sigma_y=0.05$.\label{tab:1x}}
  \footnotesize
 \begin{tabular}{@{\hspace{0pt}}c@{\hspace{3pt}}c@{\hspace{3pt}}c@{\hspace{3pt}}c@{\hspace{3pt}}c@{\hspace{3pt}}c@{\hspace{3pt}}c@{\hspace{3pt}}c@{\hspace{3pt}}c@{\hspace{3pt}}c@{\hspace{3pt}}c@{\hspace{3pt}}c@{\hspace{3pt}}c@{\hspace{3pt}}c@{\hspace{3pt}}c@{\hspace{2pt}}c@{\hspace{3pt}}c@{\hspace{3pt}}c@{\hspace{3pt}}c@{\hspace{3pt}}c@{\hspace{2pt}}c@{\hspace{3pt}}c@{\hspace{3pt}}c@{\hspace{0pt}}c@{\hspace{3pt}}c@{\hspace{3pt}}c@{\hspace{3pt}}c@{\hspace{0pt}}}
 \toprule
    &   \multicolumn{3}{c}{\textbf{Deblur (Gau)}}  && \multicolumn{3}{c}{\textbf{Deblur (uni)}} && \multicolumn{3}{c}{\textbf{Inpaint (box)}} && \multicolumn{3}{c}{\textbf{Inpaint (rand)}} && \multicolumn{3}{c}{\textbf{SR (x4)}}\\
    \cline{2-4} \cline{6-8} \cline{10-12} \cline{14-16} \cline{18-20}
     $\xi$ & PSNR  & SSIM & \makesamewidth[c]{SSIM}{LPIPS} && PSNR  & SSIM & \makesamewidth[c]{SSIM}{LPIPS} && PSNR  & SSIM & \makesamewidth[c]{SSIM}{LPIPS} && PSNR  & SSIM & \makesamewidth[c]{SSIM}{LPIPS} && PSNR  & SSIM & \makesamewidth[c]{SSIM}{LPIPS}\\
   \toprule
0.0   & 9.42  & 0.044 & 0.808 &  & 8.35  & 0.030  & 0.804 &  & 21.74 & 0.627 & 0.233 &  & 19.88 & 0.335 & 0.455 &  & 21.14 & 0.294 & 0.513 \\
0.2 & 16.41 & 0.178 & 0.654 &  & 11.98 & 0.079 & 0.742 &  & 21.66 & 0.667 & 0.216 &  & 24.39 & 0.663 & 0.265 &  & 25.14 & 0.518 & 0.405 \\
0.4 & 24.01 & 0.543 & 0.374 &  & 16.95 & 0.199 & 0.624 &  & 22.65 & 0.705 & 0.199 &  & 25.25 & 0.722 & 0.232 &  & 26.63 & 0.625 & 0.348 \\
0.6 & 26.39 & 0.719 & 0.255 &  & 25.12 & 0.548 & 0.341 &  & 23.17 & 0.734 & 0.187 &  & 25.79 & 0.755 & 0.215 &  & 27.92 & 0.720  & 0.284 \\
0.8 & 27.07 & 0.759 & 0.251 &  & 28.70  & 0.779 & 0.205 &  & 23.96 & 0.767 & 0.175 &  & 25.75 & 0.766 & 0.218 &  & 29.06 & 0.798 & 0.225 \\
1.0   & 27.29 & 0.771 & 0.284 &  & 29.41 & 0.820  & 0.235 &  & 24.54 & 0.813 & 0.157 &  & 26.05 & 0.773 & 0.269 &  & 29.74 & 0.835 & 0.217 \\
\toprule
\end{tabular} %
\end{table}
\normalsize

\begin{table}[!htp]
  \centering 
  \caption{Quantitative results of our method with different $\xi$ on the ImageNet dataset. Inputs have an additive noise with $\sigma_y=0.05$.\label{tab:2x}}
  \footnotesize
 \begin{tabular}{@{\hspace{0pt}}c@{\hspace{3pt}}c@{\hspace{3pt}}c@{\hspace{3pt}}c@{\hspace{3pt}}c@{\hspace{3pt}}c@{\hspace{3pt}}c@{\hspace{3pt}}c@{\hspace{3pt}}c@{\hspace{3pt}}c@{\hspace{3pt}}c@{\hspace{3pt}}c@{\hspace{3pt}}c@{\hspace{3pt}}c@{\hspace{3pt}}c@{\hspace{2pt}}c@{\hspace{3pt}}c@{\hspace{3pt}}c@{\hspace{3pt}}c@{\hspace{3pt}}c@{\hspace{2pt}}c@{\hspace{3pt}}c@{\hspace{3pt}}c@{\hspace{0pt}}c@{\hspace{3pt}}c@{\hspace{3pt}}c@{\hspace{3pt}}c@{\hspace{0pt}}}
 \toprule
    &   \multicolumn{3}{c}{\textbf{Deblur (Gau)}}  && \multicolumn{3}{c}{\textbf{Deblur (uni)}} && \multicolumn{3}{c}{\textbf{Inpaint (box)}} && \multicolumn{3}{c}{\textbf{Inpaint (rand)}} && \multicolumn{3}{c}{\textbf{SR (x4)}}\\
    \cline{2-4} \cline{6-8} \cline{10-12} \cline{14-16} \cline{18-20}
     $\xi$ & PSNR  & SSIM & \makesamewidth[c]{SSIM}{LPIPS} && PSNR  & SSIM & \makesamewidth[c]{SSIM}{LPIPS} && PSNR  & SSIM & \makesamewidth[c]{SSIM}{LPIPS} && PSNR  & SSIM & \makesamewidth[c]{SSIM}{LPIPS} && PSNR  & SSIM & \makesamewidth[c]{SSIM}{LPIPS}\\
   \toprule
0.0   & 8.31  & 0.027 & 0.790 &  & 7.83  & 0.029 & 0.764 &  & 16.36 & 0.558 & 0.304 &  & 12.69 & 0.054 & 0.669 &  & 16.14 & 0.128 & 0.605 \\
0.2 & 10.03 & 0.039 & 0.751 &  & 9.70  & 0.050 & 0.739 &  & 18.54 & 0.589 & 0.268 &  & 18.73 & 0.208 & 0.538 &  & 19.31 & 0.215 & 0.529 \\
0.4 & 12.27 & 0.062 & 0.714 &  & 11.05 & 0.064 & 0.718 &  & 19.72 & 0.619 & 0.250 &  & 21.39 & 0.353 & 0.462 &  & 22.56 & 0.359 & 0.464 \\
0.6 & 19.58 & 0.221 & 0.567 &  & 12.81 & 0.085 & 0.687 &  & 20.53 & 0.651 & 0.249 &  & 23.12 & 0.486 & 0.414 &  & 24.55 & 0.491 & 0.434 \\
0.8 & 24.89 & 0.559 & 0.404 &  & 17.69 & 0.180 & 0.592 &  & 21.01 & 0.680 & 0.259 &  & 24.08 & 0.578 & 0.391 &  & 25.62 & 0.572 & 0.399 \\
1.0   & 25.28 & 0.587 & 0.471 &  & 26.86 & 0.654 & 0.392 &  & 21.76 & 0.726 & 0.248 &  & 24.20 & 0.592 & 0.453 &  & 26.88 & 0.665 & 0.365 \\
\toprule
\end{tabular} %
\end{table}
\normalsize

\begin{figure}[!htp]
   \centering 
  \begin{tabular}{c@{\hspace*{1pt}}c@{\hspace*{1pt}}c@{\hspace*{1pt}}c@{\hspace*{1pt}}c@{\hspace*{1pt}}c@{\hspace*{1pt}}c@{\hspace*{1pt}}c@{\hspace*{1pt}}c@{\hspace*{1pt}}c@{\hspace*{2pt}}c@{\hspace*{2pt}}c@{\hspace*{2pt}}c@{\hspace*{1pt}}}
 \rotatebox[origin=c]{90}{\small{SR}} & \raisebox{-0.5\height}{\includegraphics[width = 0.118\textwidth]{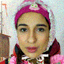}}
 & \raisebox{-0.5\height}{\includegraphics[width = 0.118\textwidth]{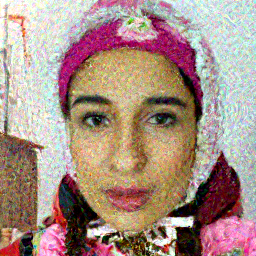}}
 & \raisebox{-0.5\height}{\includegraphics[width = 0.118\textwidth]{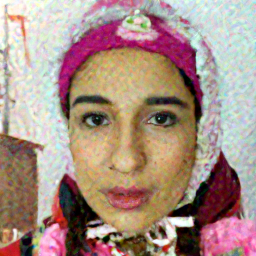}}
 & \raisebox{-0.5\height}{\includegraphics[width = 0.118\textwidth]{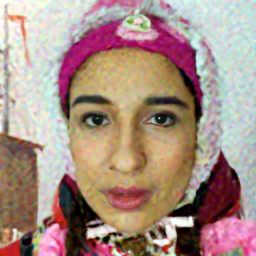}}
 & \raisebox{-0.5\height}{\includegraphics[width = 0.118\textwidth]{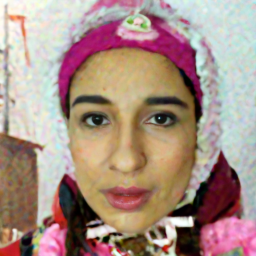}}
 & \raisebox{-0.5\height}{\includegraphics[width = 0.118\textwidth]{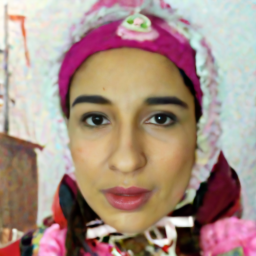}}
 & \raisebox{-0.5\height}{\includegraphics[width = 0.118\textwidth]{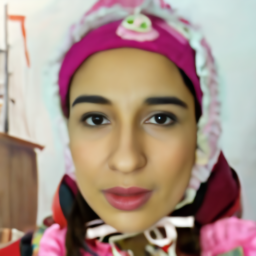}}
    & \raisebox{-0.5\height}{\includegraphics[width = 0.118\textwidth]{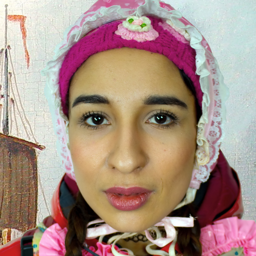}}\\
 \rotatebox[origin=c]{90}{\small{Deblur (Gau)}} & \raisebox{-0.5\height}{\includegraphics[width = 0.118\textwidth]{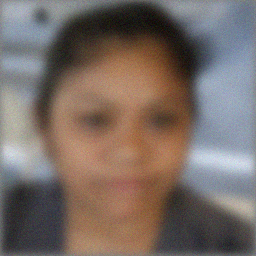}}
 & \raisebox{-0.5\height}{\includegraphics[width = 0.118\textwidth]{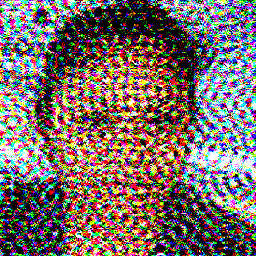}}
 & \raisebox{-0.5\height}{\includegraphics[width = 0.118\textwidth]{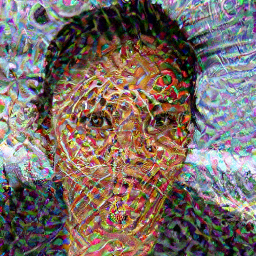}}
 & \raisebox{-0.5\height}{\includegraphics[width = 0.118\textwidth]{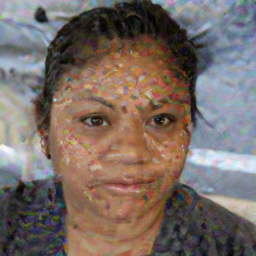}}
 & \raisebox{-0.5\height}{\includegraphics[width = 0.118\textwidth]{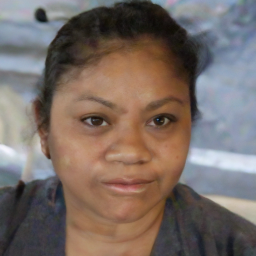}}
 & \raisebox{-0.5\height}{\includegraphics[width = 0.118\textwidth]{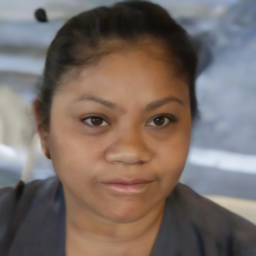}}
 & \raisebox{-0.5\height}{\includegraphics[width = 0.118\textwidth]{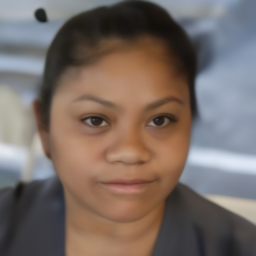}}
    & \raisebox{-0.5\height}{\includegraphics[width = 0.118\textwidth]{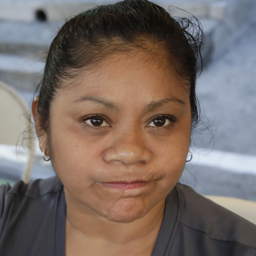}}\\
 \rotatebox[origin=c]{90}{\small{Deblur (uni)}} & \raisebox{-0.5\height}{\includegraphics[width = 0.118\textwidth]{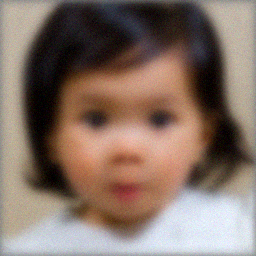}}
 & \raisebox{-0.5\height}{\includegraphics[width = 0.118\textwidth]{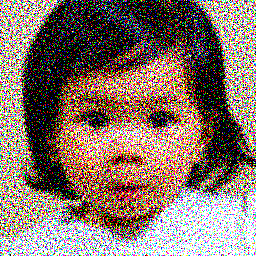}}
 & \raisebox{-0.5\height}{\includegraphics[width = 0.118\textwidth]{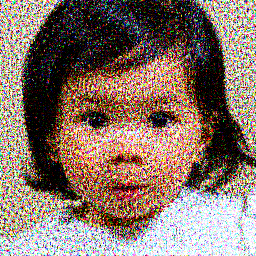}}
 & \raisebox{-0.5\height}{\includegraphics[width = 0.118\textwidth]{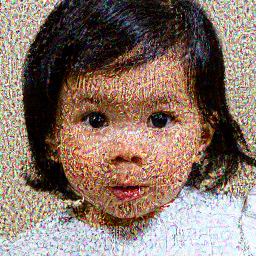}}
 & \raisebox{-0.5\height}{\includegraphics[width = 0.118\textwidth]{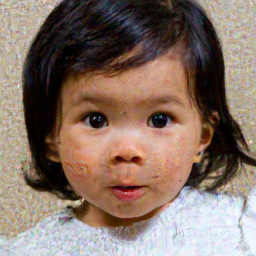}}
 & \raisebox{-0.5\height}{\includegraphics[width = 0.118\textwidth]{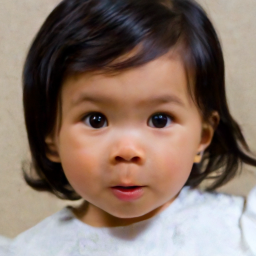}}
 & \raisebox{-0.5\height}{\includegraphics[width = 0.118\textwidth]{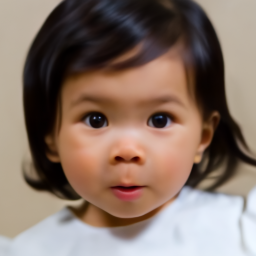}}
    & \raisebox{-0.5\height}{\includegraphics[width = 0.118\textwidth]{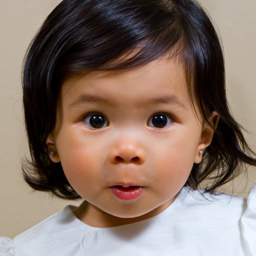}}\\
 \rotatebox[origin=c]{90}{\small{Inp (box)}} & \raisebox{-0.5\height}{\includegraphics[width = 0.118\textwidth]{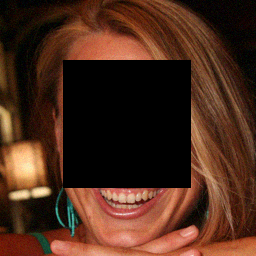}}
 & \raisebox{-0.5\height}{\includegraphics[width = 0.118\textwidth]{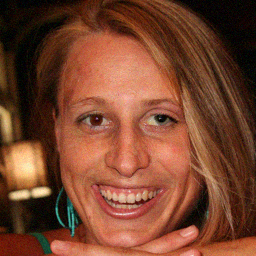}}
 & \raisebox{-0.5\height}{\includegraphics[width = 0.118\textwidth]{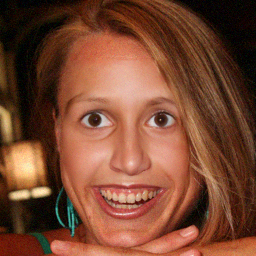}}
 & \raisebox{-0.5\height}{\includegraphics[width = 0.118\textwidth]{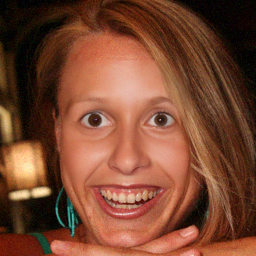}}
 & \raisebox{-0.5\height}{\includegraphics[width = 0.118\textwidth]{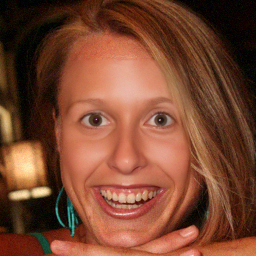}}
 & \raisebox{-0.5\height}{\includegraphics[width = 0.118\textwidth]{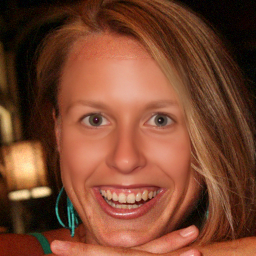}}
 & \raisebox{-0.5\height}{\includegraphics[width = 0.118\textwidth]{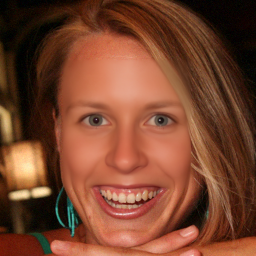}}
    & \raisebox{-0.5\height}{\includegraphics[width = 0.118\textwidth]{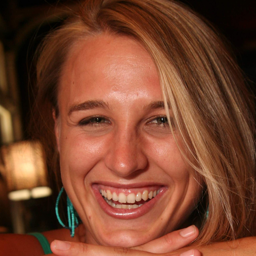}}\\
 \rotatebox[origin=c]{90}{\small{Inp (rand.)}} & \raisebox{-0.5\height}{\includegraphics[width = 0.118\textwidth]{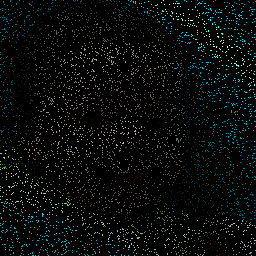}}
 & \raisebox{-0.5\height}{\includegraphics[width = 0.118\textwidth]{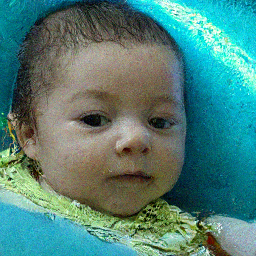}}
 & \raisebox{-0.5\height}{\includegraphics[width = 0.118\textwidth]{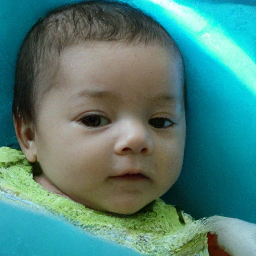}}
 & \raisebox{-0.5\height}{\includegraphics[width = 0.118\textwidth]{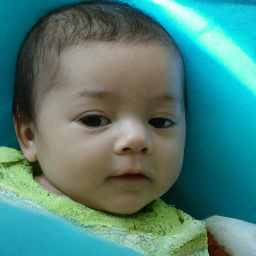}}
 & \raisebox{-0.5\height}{\includegraphics[width = 0.118\textwidth]{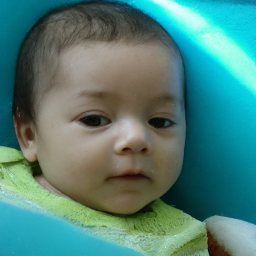}}
 & \raisebox{-0.5\height}{\includegraphics[width = 0.118\textwidth]{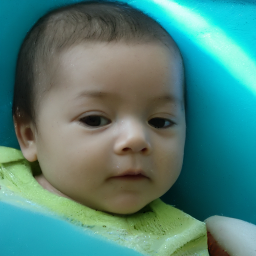}}
 & \raisebox{-0.5\height}{\includegraphics[width = 0.118\textwidth]{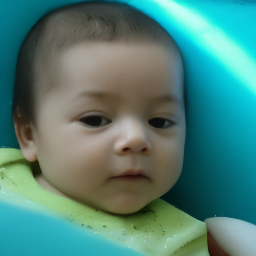}}
 & \raisebox{-0.5\height}{\includegraphics[width = 0.118\textwidth]{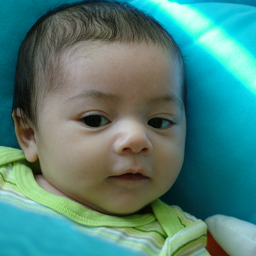}}\\ 
 & \small{Input} & \small{$\xi=0.0$} & \small{$\xi=0.2$} & \small{$\xi=0.4$} 
 & \small{$\xi=0.6$} & \small{$\xi=0.8$} & 
 \small{$\xi=1.0$} 
 & \small{GT}\\
  \end{tabular}
  \caption{Visualization of the results from our method with different $\xi$ for Gaussian noisy linear problems on the FFHQ $256\times 256$. Gaussian noise ($\sigma_y=0.05$) is added to the measurement. \label{fig:ffhq_gauss_x}} %
 \end{figure}
 
 \begin{figure}[!htbp]
   \centering 
  \begin{tabular}{c@{\hspace*{1pt}}c@{\hspace*{1pt}}c@{\hspace*{1pt}}c@{\hspace*{1pt}}c@{\hspace*{1pt}}c@{\hspace*{1pt}}c@{\hspace*{1pt}}c@{\hspace*{1pt}}c@{\hspace*{1pt}}c@{\hspace*{2pt}}c@{\hspace*{2pt}}c@{\hspace*{2pt}}c@{\hspace*{1pt}}}
 \rotatebox[origin=c]{90}{\small{SR}} &\raisebox{-0.5\height}{\includegraphics[width = 0.118\textwidth]{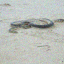}}
 & \raisebox{-0.5\height}{\includegraphics[width = 0.118\textwidth]{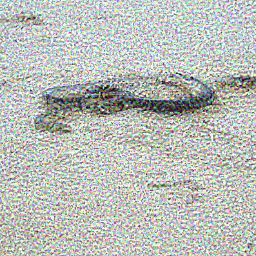}}
 & \raisebox{-0.5\height}{\includegraphics[width = 0.118\textwidth]{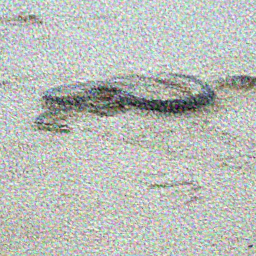}}
 & \raisebox{-0.5\height}{\includegraphics[width = 0.118\textwidth]{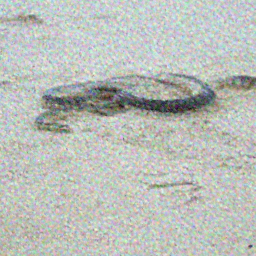}}
 & \raisebox{-0.5\height}{\includegraphics[width = 0.118\textwidth]{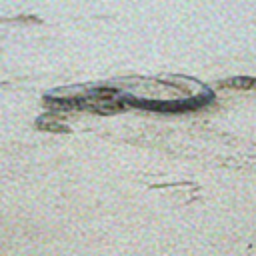}}
 & \raisebox{-0.5\height}{\includegraphics[width = 0.118\textwidth]{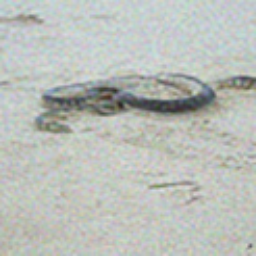}}
 & \raisebox{-0.5\height}{\includegraphics[width = 0.118\textwidth]{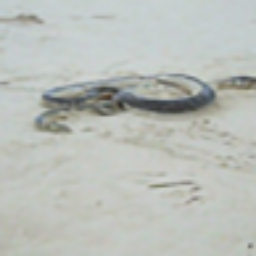}}
    & \raisebox{-0.5\height}{\includegraphics[width = 0.118\textwidth]{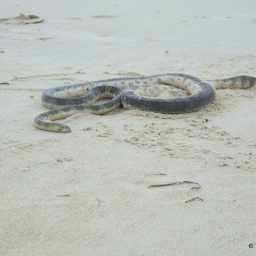}}\\
 \rotatebox[origin=c]{90}{\small{Deblur (Gau)}} &   \raisebox{-0.5\height}{\includegraphics[width = 0.118\textwidth]{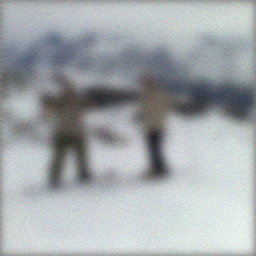}}
 & \raisebox{-0.5\height}{\includegraphics[width = 0.118\textwidth]{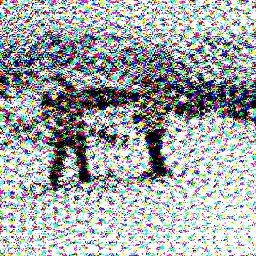}}
 & \raisebox{-0.5\height}{\includegraphics[width = 0.118\textwidth]{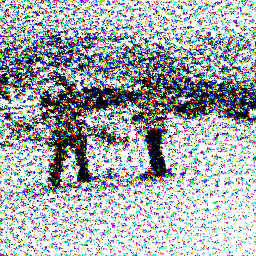}}
 & \raisebox{-0.5\height}{\includegraphics[width = 0.118\textwidth]{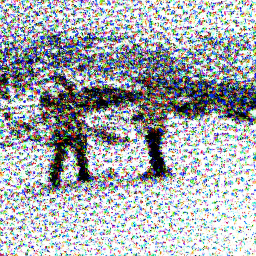}}
 & \raisebox{-0.5\height}{\includegraphics[width = 0.118\textwidth]{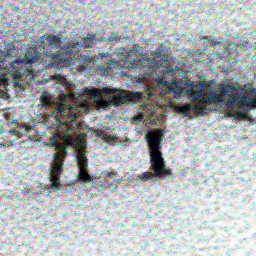}}
 & \raisebox{-0.5\height}{\includegraphics[width = 0.118\textwidth]{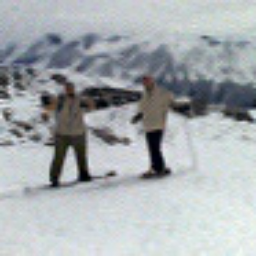}}
 & \raisebox{-0.5\height}{\includegraphics[width = 0.118\textwidth]{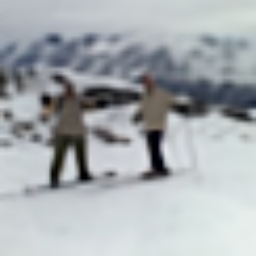}}
    & \raisebox{-0.5\height}{\includegraphics[width = 0.118\textwidth]{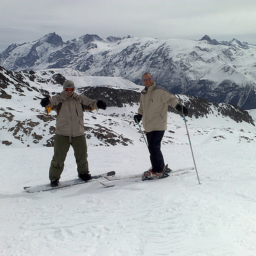}}\\
 \rotatebox[origin=c]{90}{\small{Deblur (uni)}} &\raisebox{-0.5\height}{\includegraphics[width = 0.118\textwidth]{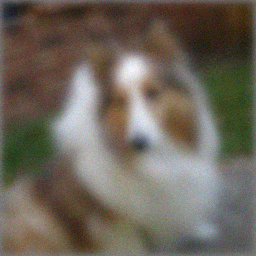}}
 & \raisebox{-0.5\height}{\includegraphics[width = 0.118\textwidth]{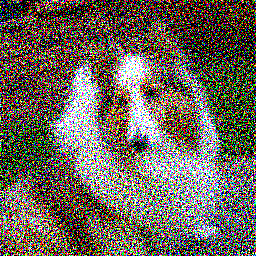}}
 & \raisebox{-0.5\height}{\includegraphics[width = 0.118\textwidth]{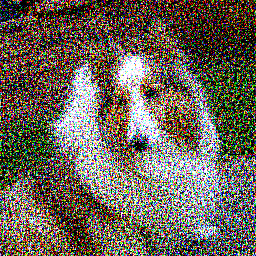}}
 & \raisebox{-0.5\height}{\includegraphics[width = 0.118\textwidth]{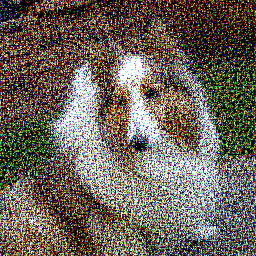}}
 & \raisebox{-0.5\height}{\includegraphics[width = 0.118\textwidth]{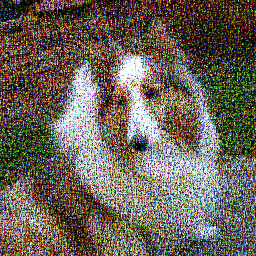}}
 & \raisebox{-0.5\height}{\includegraphics[width = 0.118\textwidth]{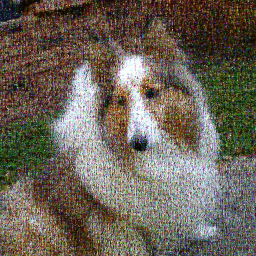}}
 & \raisebox{-0.5\height}{\includegraphics[width = 0.118\textwidth]{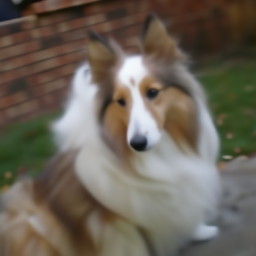}}
    & \raisebox{-0.5\height}{\includegraphics[width = 0.118\textwidth]{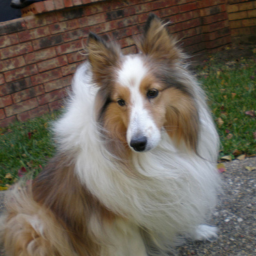}}\\
 \rotatebox[origin=c]{90}{\small{Inp (box)}} &   \raisebox{-0.5\height}{\includegraphics[width = 0.118\textwidth]{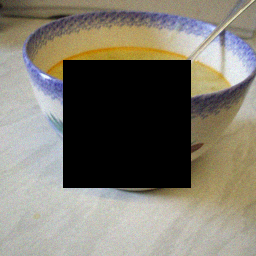}}
 & \raisebox{-0.5\height}{\includegraphics[width = 0.118\textwidth]{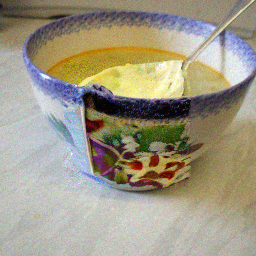}}
 & \raisebox{-0.5\height}{\includegraphics[width = 0.118\textwidth]{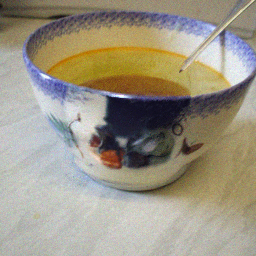}}
 & \raisebox{-0.5\height}{\includegraphics[width = 0.118\textwidth]{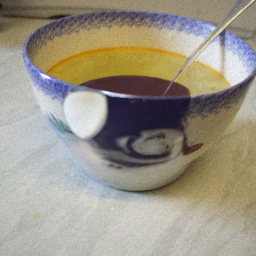}}
 & \raisebox{-0.5\height}{\includegraphics[width = 0.118\textwidth]{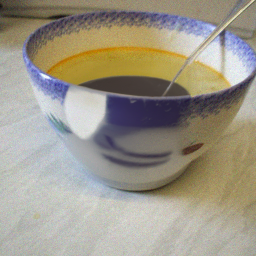}}
 & \raisebox{-0.5\height}{\includegraphics[width = 0.118\textwidth]{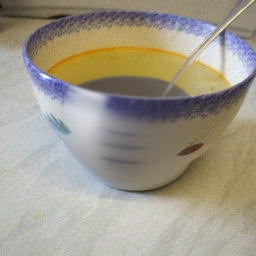}}
 & \raisebox{-0.5\height}{\includegraphics[width = 0.118\textwidth]{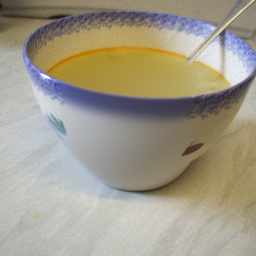}}
    & \raisebox{-0.5\height}{\includegraphics[width = 0.118\textwidth]{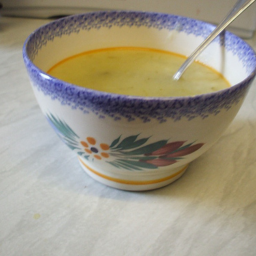}}\\
 \rotatebox[origin=c]{90}{\small{Inp (rand.)}} &      \raisebox{-0.5\height}{\includegraphics[width = 0.118\textwidth]{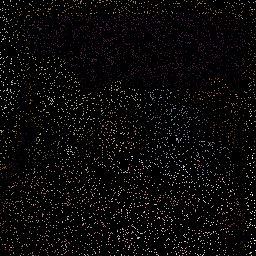}}
 & \raisebox{-0.5\height}{\includegraphics[width = 0.118\textwidth]{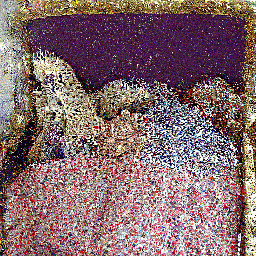}}
 & \raisebox{-0.5\height}{\includegraphics[width = 0.118\textwidth]{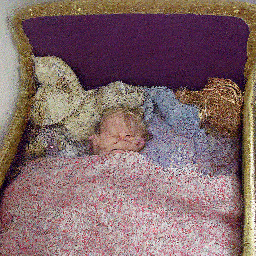}}
 & \raisebox{-0.5\height}{\includegraphics[width = 0.118\textwidth]{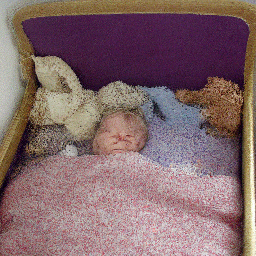}}
 & \raisebox{-0.5\height}{\includegraphics[width = 0.118\textwidth]{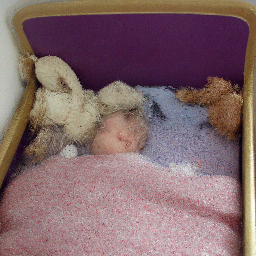}}
 & \raisebox{-0.5\height}{\includegraphics[width = 0.118\textwidth]{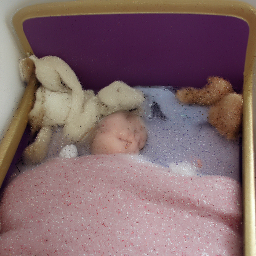}}
 & \raisebox{-0.5\height}{\includegraphics[width = 0.118\textwidth]{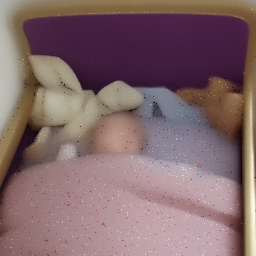}}
 & \raisebox{-0.5\height}{\includegraphics[width = 0.118\textwidth]{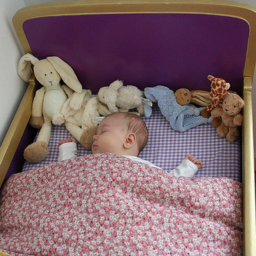}}\\
 & \small{Input} & \small{$\xi=0.0$} & \small{$\xi=0.2$} & \small{$\xi=0.4$} 
 & \small{$\xi=0.6$} & \small{$\xi=0.8$} & 
 \small{$\xi=1.0$} 
 & \small{GT}\\
  \end{tabular}%
 \caption{Visualization of the results from our method with different $\xi$ for Gaussian noisy linear problems on the ImageNet $256\times 256$. Gaussian noise ($\sigma_y=0.05$) is added to the measurement. \label{fig:img_gauss_x}} %
 \end{figure}

\paragraph{Quantitative results of the compared methods}

 The quantitative comparison of different algorithms in terms of the three metrics is shown in Table~\ref{tab:1m} and~\ref{tab:2m} for FFHQ and ImageNet respectively. The best result is in \textcolor{blue}{blue}, and the second best is \underline{underlined}.   %

 \begin{table}[!htp]
  \centering 
  \caption{Quantitative results of different methods on the FFHQ dataset. Inputs have an additive noise with $\sigma_y=0.05$.\label{tab:1m}}
  \footnotesize
 \begin{tabular}{@{\hspace{0pt}}c@{\hspace{3pt}}c@{\hspace{3pt}}c@{\hspace{3pt}}c@{\hspace{3pt}}c@{\hspace{3pt}}c@{\hspace{3pt}}c@{\hspace{3pt}}c@{\hspace{3pt}}c@{\hspace{3pt}}c@{\hspace{3pt}}c@{\hspace{3pt}}c@{\hspace{3pt}}c@{\hspace{3pt}}c@{\hspace{3pt}}c@{\hspace{2pt}}c@{\hspace{3pt}}c@{\hspace{3pt}}c@{\hspace{3pt}}c@{\hspace{3pt}}c@{\hspace{2pt}}c@{\hspace{3pt}}c@{\hspace{3pt}}c@{\hspace{0pt}}c@{\hspace{3pt}}c@{\hspace{3pt}}c@{\hspace{3pt}}c@{\hspace{0pt}}}
 \toprule
    &   \multicolumn{3}{c}{\textbf{Deblur (Gau)}}  && \multicolumn{3}{c}{\textbf{Deblur (uni)}} && \multicolumn{3}{c}{\textbf{Inpaint (box)}} && \multicolumn{3}{c}{\textbf{Inpaint (rand)}} && \multicolumn{3}{c}{\textbf{SR (x4)}}\\
    \cline{2-4} \cline{6-8} \cline{10-12} \cline{14-16} \cline{18-20}
     Method & PSNR  & SSIM & \makesamewidth[c]{SSIM}{LPIPS} && PSNR  & SSIM & \makesamewidth[c]{SSIM}{LPIPS} && PSNR  & SSIM & \makesamewidth[c]{SSIM}{LPIPS} && PSNR  & SSIM & \makesamewidth[c]{SSIM}{LPIPS} && PSNR  & SSIM & \makesamewidth[c]{SSIM}{LPIPS}\\
   \toprule
   L2TV & 25.11 & 0.675 & 0.429 &  & 26.39 & 0.714 & 0.396 &  & 18.94 & 0.721 & 0.318 &  & 15.30 & 0.385 & 0.608 &  & 27.33 & 0.745 & 0.367 \\
   \toprule
DDRM     & 23.61       & 0.674 & 0.271 && 27.10        & 0.764 & \ul{0.237} && 24.10     & 0.834 & \ul{0.148} && 21.88       & 0.634 & 0.346 && 28.01 & 0.790  & 0.231 \\
   DDNM+    & 10.06       & 0.051 & 0.756 && 28.63      & 0.793 & 0.256 && \ul{25.12}    & \bl{0.850}  & \bl{0.148} && 24.40        & 0.716 & 0.317 && 28.82 & 0.807 & 0.251 \\
   \toprule
DPS      & 23.74       & 0.635 & \ul{0.266} && 24.38       & 0.657 & 0.259 && 23.54    & 0.744 & 0.230  && 25.54       & 0.732 & \bl{0.234} && 25.75 & 0.708 & 0.246 \\
$\Pi$GDM & 25.23       & 0.713 & 0.307 && \ul{29.15}       & 0.807 & 0.249 && 22.56    & \ul{0.838} & 0.178 && \bl{26.20}        & \ul{0.758} & 0.297 && \ul{29.47} & \ul{0.823} & 0.233 \\
DMPS     & \ul{26.18}       & \ul{0.727} & \bl{0.250}  && 28.41       & 0.789 & \bl{0.216} && 15.62    & 0.680  & 0.318 && 13.52       & 0.280  & 0.625 && 28.45 & 0.799 & \bl{0.205} \\
Ours     & \bl{27.35}       & \bl{0.761} & 0.286 && \bl{29.19}       & 0.809 & 0.240  && \bl{25.13}    & 0.808 & 0.167 && \ul{25.92}       & \bl{0.762} & \ul{0.271} && \bl{29.52} & \bl{0.825} & \ul{0.224} \\
\toprule
\end{tabular} %
\end{table}
\normalsize

\begin{table}[!htp]
   \centering 
   \caption{Quantitative results of different methods on the ImageNet dataset. Inputs have an additive noise with $\sigma_y=0.05$.\label{tab:2m}}
   \footnotesize
 \begin{tabular}{@{\hspace{0pt}}c@{\hspace{3pt}}c@{\hspace{3pt}}c@{\hspace{3pt}}c@{\hspace{3pt}}c@{\hspace{3pt}}c@{\hspace{3pt}}c@{\hspace{3pt}}c@{\hspace{3pt}}c@{\hspace{3pt}}c@{\hspace{3pt}}c@{\hspace{3pt}}c@{\hspace{3pt}}c@{\hspace{3pt}}c@{\hspace{3pt}}c@{\hspace{2pt}}c@{\hspace{3pt}}c@{\hspace{3pt}}c@{\hspace{3pt}}c@{\hspace{3pt}}c@{\hspace{2pt}}c@{\hspace{3pt}}c@{\hspace{3pt}}c@{\hspace{0pt}}c@{\hspace{3pt}}c@{\hspace{3pt}}c@{\hspace{3pt}}c@{\hspace{0pt}}}
 \toprule
    &   \multicolumn{3}{c}{\textbf{Deblur (Gau)}}  && \multicolumn{3}{c}{\textbf{Deblur (uni)}} && \multicolumn{3}{c}{\textbf{Inpaint (box)}} && \multicolumn{3}{c}{\textbf{Inpaint (rand)}} && \multicolumn{3}{c}{\textbf{SR (x4)}}\\
    \cline{2-4} \cline{6-8} \cline{10-12} \cline{14-16} \cline{18-20}
     Method & PSNR  & SSIM & \makesamewidth[c]{SSIM}{LPIPS} && PSNR  & SSIM & \makesamewidth[c]{SSIM}{LPIPS} && PSNR  & SSIM & \makesamewidth[c]{SSIM}{LPIPS} && PSNR  & SSIM & \makesamewidth[c]{SSIM}{LPIPS} && PSNR  & SSIM & \makesamewidth[c]{SSIM}{LPIPS}\\
   \toprule
   L2TV & 23.03 & 0.562 & 0.471 &  & 24.18 & 0.612 & 0.436 &  & 17.92 & 0.699 & 0.290 &  & 14.92 & 0.319 & 0.612 &  & 25.07 & 0.652 & 0.401 \\
   \toprule
DDRM     & 21.66 & 0.557 & \bl{0.391} &  & 25.08 & 0.677 & 0.326 &  & 20.01 & \ul{0.773} & \bl{0.219} &  & 19.65 & 0.444 & 0.529 &  & 25.66 & 0.700   & \ul{0.316} \\
DDNM+    & 8.80   & 0.046 & 0.745 &  & 25.08 & 0.626 & 0.391 &  & \ul{20.94} & \bl{0.790}  & \ul{0.223} &  & 21.73 & 0.565 & 0.458 &  & 26.15 & 0.714 & 0.321 \\
   \toprule
   DPS      & 17.99 & 0.361 & 0.457 &  & 21.78 & 0.505 & 0.380  &  & 19.08 & 0.655 & 0.330  &  & \bl{23.61} & \bl{0.648} & \bl{0.307} &  & 24.14 & 0.622 & 0.341 \\
$\Pi$GDM & 23.61 & 0.594 & 0.411 &  & \ul{26.76} & \ul{0.718} & 0.324 &  & 19.06 & 0.761 & 0.268 &  & \ul{23.46} & 0.617 & 0.426 &  & \bl{26.57} & \bl{0.728} & \bl{0.311} \\
DMPS     & \ul{23.92} & \ul{0.605} & 0.406 &  & 25.89 & 0.689 & \ul{0.320}  &  & 19.30  & 0.759 & 0.251 &  & 12.39 & 0.193 & 0.699 &  & 26.05 & 0.706 & 0.319 \\
Ours     & \bl{24.33} & \bl{0.622} & \ul{0.404} &  & \bl{26.84} & \bl{0.724} & \bl{0.306} &  & \bl{20.97} & 0.740  & 0.235 &  & 23.13 & \ul{0.624} & \ul{0.402} &  & \ul{26.50}  & \ul{0.726} & 0.299\\
\toprule
\end{tabular} %
\end{table}
\normalsize

As shown in the tables, though there is no overwhelming best performer among the diffusion model based methods across the tested tasks, the diffusion-based method remarkably outperforms the traditional TV regularized optimization for both the two considered datasets, showing the promising performance of image prior from diffusion model over TV regularization. Except for the two inpainting tasks, Ours, $\Pi$GDM and DMPS are highly related except for minor difference among them. Their performance is comparable, and ours achieves the best or second best results. DPS outperforms other methods in some tasks, this is attributed to its much longer 1k diffusion-step running. As a comparison, Ours, $\Pi$GDM and DMPS performed 100 diffusion steps. Notice that for inpainting tasks, Ours outperforms DMPS remarkably. This fact illustrates our MAP-based measurement guided diffusion is theoretically sound than the derivation of DMPS. There is another striking phenomenon, with careful tuning the parameters of DDNM+, deblurring with Gaussian kernel yields highly-corrupted restoration, while deblurring with uniform kernel works satisfyingly. For DPS, the results are distorted, where the measurement consistency is not fully achieved. See Figures~\ref{fig:ffhq_gauss} and~\ref{fig:img_gauss} for the visualizations. Restored image of L2TV shows the staircase artifacts and lacks of image details for all the tasks. For image inpainting with box mask, L2TV fills the box mask with a flat value, which contains no faithful content, for diffusion model based methods, ours produces much innocent restoration than the DMPS method. The measurement inconsistency of DPS to the ground truth can be seen in the following area: the hair in forehead, the background details.

\begin{remark}[Why DMPS fails for inpainting?]
  As we shown, though DMPS is related to our MAP-based method, they produced such fundamentally different results for the inpainting tasks. For the inpainting task, the measurement did not provide any information of the masked part of the image. To accomplish the inpainting, the prior from diffusion model plays an important role for the task. To differentiate the performance from ours and DMPS, let us compare the updated $\hat{\bm{x}}_0(\bm{x}_t,\bm{y})$ from the DMPS and our method. DMPS takes the update
\begin{equation}
  \hat{\bm{x}}_0(\bm{x}_t,\bm{y}):=(\sigma_y^2\bm{I}+\frac{\sigma_t^2}{\alpha_t^2}\bm{A}^T\bm{A})^{-1}(\sigma_y^2(\frac{\bm{x}_t}{\alpha_t})+\frac{\sigma_t^2}{\alpha_t^2}\bm{A}^T\bm{y}).
\end{equation}
While ours takes the update
\begin{equation}
  \hat{\bm{x}}_0(\bm{x}_t,\bm{y}):=(\sigma_y^2\bm{I}+\frac{\sigma_t^2}{\alpha_t^2}\bm{A}^T\bm{A})^{-1}(\sigma_y^2(\frac{\bm{x}_t+\sigma_t^2s_{\theta}(\bm{x}_t,t)}{\alpha_t})+\frac{\sigma_t^2}{\alpha_t^2}\bm{A}^T\bm{y}).
\end{equation}
DMPS filled the masked part using a noisy image $\bm{x}_t$, while ours filled the masked part using a ``denoised" image $\frac{\bm{x}_t+\sigma_t^2s_{\theta}(\bm{x}_t,t)}{\alpha_t}$, which contains more image prior information. 
\end{remark}
 
\noindent{\bf The running time comparison.} \ \
See Table~\ref{tab:time} for the running time for single image on the SR $\times 4$ task. All the codes are tested on the same machine equipped with a 3090 RTX GPU. Compared to the DPS~\cite{chung2022diffusion}, $\Pi$GDM, our MAP-based measurement guided diffusion sampling avoids the backpropagation through the diffusion model, hence it accelerates the computation. It saves memory as there is no backpropagation through the network.

\begin{table}[!htp] %
  \centering
  \caption{Comparison of running time in second.\label{tab:time}}
\begin{tabular}{@{\hspace*{0pt}}c@{\hspace*{2pt}}c@{\hspace*{2pt}}c@{\hspace*{2pt}}c@{\hspace*{2pt}}c@{\hspace*{2pt}}c@{\hspace*{2pt}}c@{\hspace*{0pt}}c@{\hspace*{0pt}}c@{\hspace*{0pt}}c@{\hspace*{0pt}}c@{\hspace*{0pt}}cc@{\hspace*{0pt}}cc@{\hspace*{0pt}}}
    \toprule
Method & DDRM & DDNM$+$ & $\Pi$GDM & DPS  & Ours \\\midrule
    Time   & 0.9  & 4.5  & 6.3 & 88.1 & 4.1   \\
    \toprule
       \end{tabular} %
     \end{table}

\begin{figure}[!htp]
   \centering 
  \begin{tabular}{c@{\hspace*{1pt}}c@{\hspace*{1pt}}c@{\hspace*{1pt}}c@{\hspace*{1pt}}c@{\hspace*{1pt}}c@{\hspace*{1pt}}c@{\hspace*{1pt}}c@{\hspace*{1pt}}c@{\hspace*{1pt}}c@{\hspace*{2pt}}c@{\hspace*{2pt}}c@{\hspace*{2pt}}c@{\hspace*{1pt}}}
   \rotatebox[origin=c]{90}{\small{SR}} &  \raisebox{-25pt}{\includegraphics[width = 0.118\textwidth]{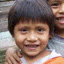}}
 & \raisebox{-25pt}{\includegraphics[width = 0.118\textwidth]{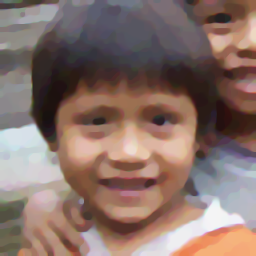}}
 & \raisebox{-25pt}{\includegraphics[width = 0.118\textwidth]{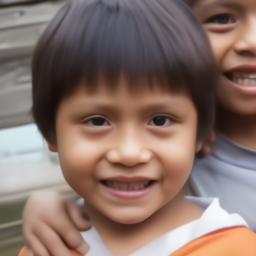}}
 & \raisebox{-25pt}{\includegraphics[width = 0.118\textwidth]{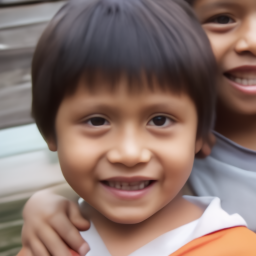}}
 & \raisebox{-25pt}{\includegraphics[width = 0.118\textwidth]{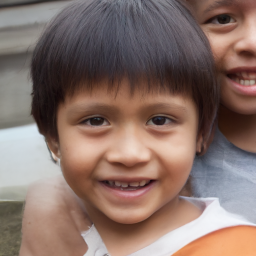}}
 & \raisebox{-25pt}{\includegraphics[width = 0.118\textwidth]{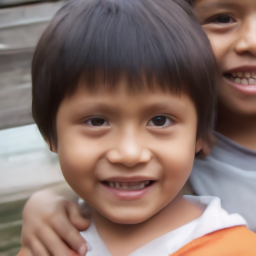}}
 & \raisebox{-25pt}{\includegraphics[width = 0.118\textwidth]{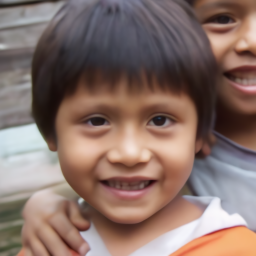}}
    & \raisebox{-25pt}{\includegraphics[width = 0.118\textwidth]{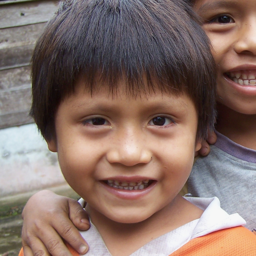}}\\
    \rotatebox[origin=c]{90}{\small{Deblur (Gau)}} & \raisebox{-25pt}{\includegraphics[width = 0.118\textwidth]{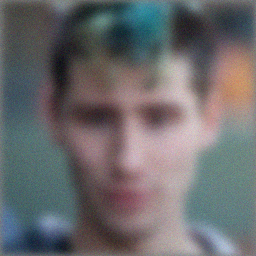}}
 & \raisebox{-25pt}{\includegraphics[width = 0.118\textwidth]{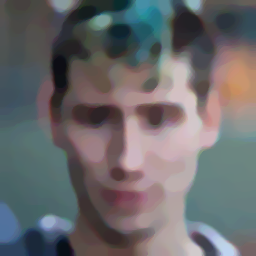}}
 & \raisebox{-25pt}{\includegraphics[width = 0.118\textwidth]{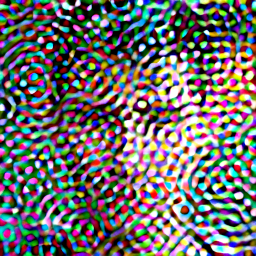}}
 & \raisebox{-25pt}{\includegraphics[width = 0.118\textwidth]{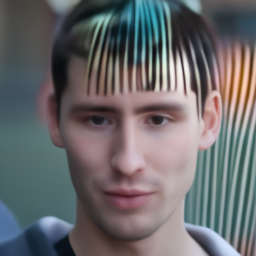}}
 & \raisebox{-25pt}{\includegraphics[width = 0.118\textwidth]{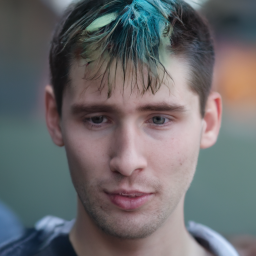}}
 & \raisebox{-25pt}{\includegraphics[width = 0.118\textwidth]{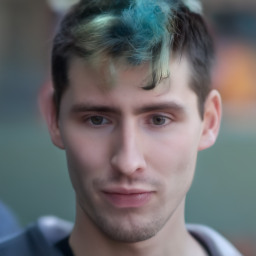}}
 & \raisebox{-25pt}{\includegraphics[width = 0.118\textwidth]{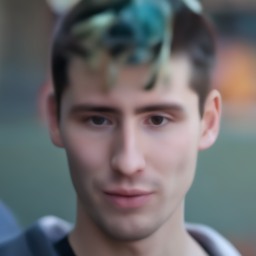}}
    & \raisebox{-25pt}{\includegraphics[width = 0.118\textwidth]{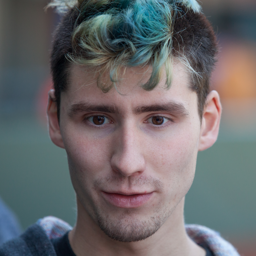}}\\
    \rotatebox[origin=c]{90}{\small{Deblur (uni)}} & \raisebox{-25pt}{\includegraphics[width = 0.118\textwidth]{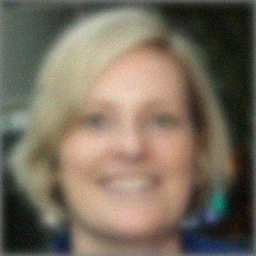}}
 & \raisebox{-25pt}{\includegraphics[width = 0.118\textwidth]{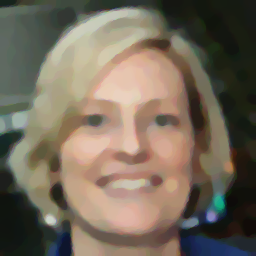}}
 & \raisebox{-25pt}{\includegraphics[width = 0.118\textwidth]{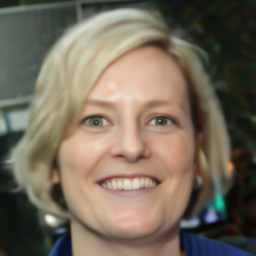}}
 & \raisebox{-25pt}{\includegraphics[width = 0.118\textwidth]{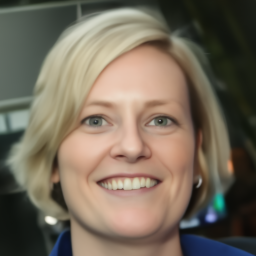}}
 & \raisebox{-25pt}{\includegraphics[width = 0.118\textwidth]{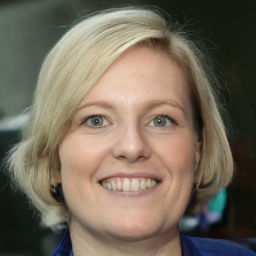}}
 & \raisebox{-25pt}{\includegraphics[width = 0.118\textwidth]{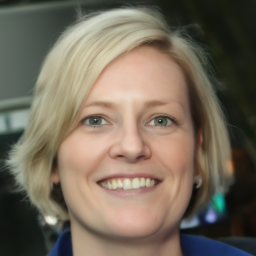}}
 & \raisebox{-25pt}{\includegraphics[width = 0.118\textwidth]{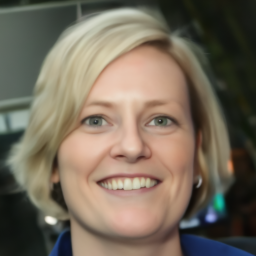}}
    & \raisebox{-25pt}{\includegraphics[width = 0.118\textwidth]{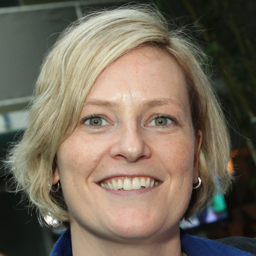}}\\
    \rotatebox[origin=c]{90}{\small{Inp (box)}} &\raisebox{-25pt}{\includegraphics[width = 0.118\textwidth]{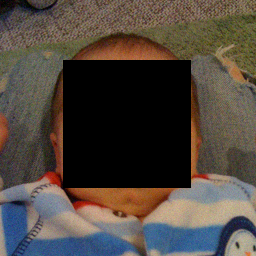}}
 & \raisebox{-25pt}{\includegraphics[width = 0.118\textwidth]{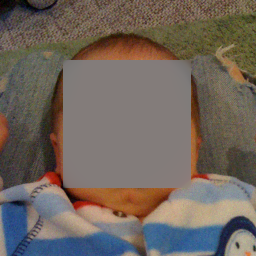}}
 & \raisebox{-25pt}{\includegraphics[width = 0.118\textwidth]{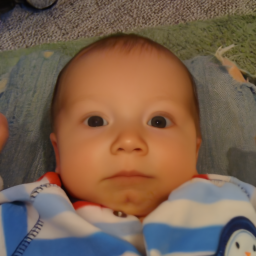}}
 & \raisebox{-25pt}{\includegraphics[width = 0.118\textwidth]{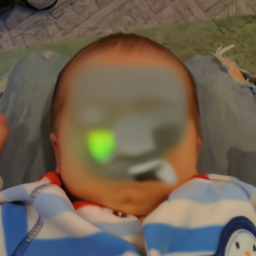}}
 & \raisebox{-25pt}{\includegraphics[width = 0.118\textwidth]{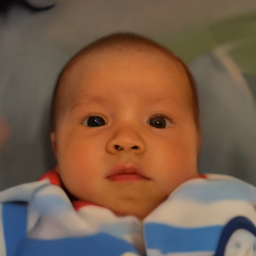}}
 & \raisebox{-25pt}{\includegraphics[width = 0.118\textwidth]{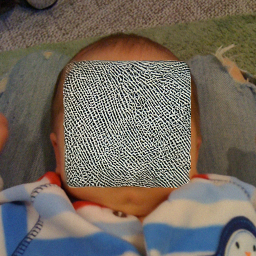}}
 & \raisebox{-25pt}{\includegraphics[width = 0.118\textwidth]{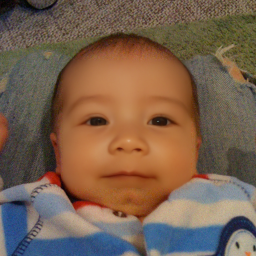}}
    & \raisebox{-25pt}{\includegraphics[width = 0.118\textwidth]{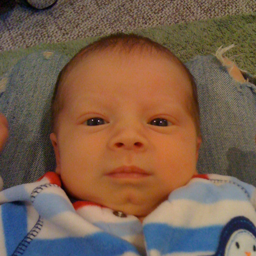}}\\
    \rotatebox[origin=c]{90}{\small{Inp (rand.)}} &  \raisebox{-25pt}{\includegraphics[width = 0.118\textwidth]{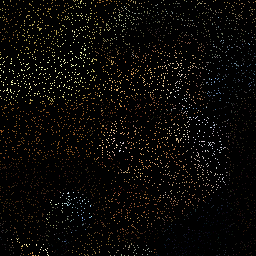}}
 & \raisebox{-25pt}{\includegraphics[width = 0.118\textwidth]{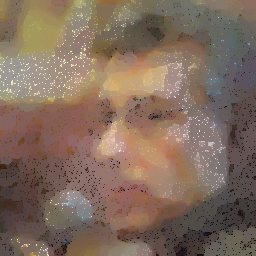}}
 & \raisebox{-25pt}{\includegraphics[width = 0.118\textwidth]{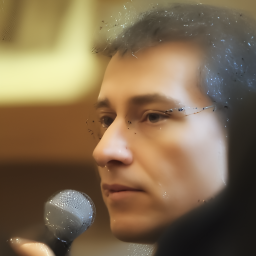}}
 & \raisebox{-25pt}{\includegraphics[width = 0.118\textwidth]{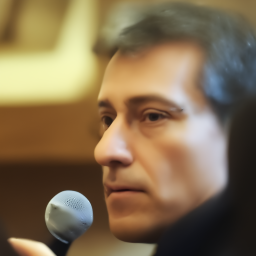}}
 & \raisebox{-25pt}{\includegraphics[width = 0.118\textwidth]{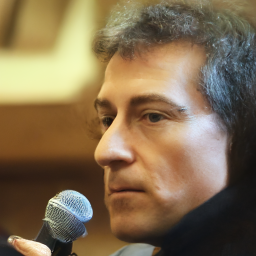}}
 & \raisebox{-25pt}{\includegraphics[width = 0.118\textwidth]{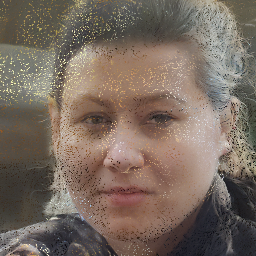}}
 & \raisebox{-25pt}{\includegraphics[width = 0.118\textwidth]{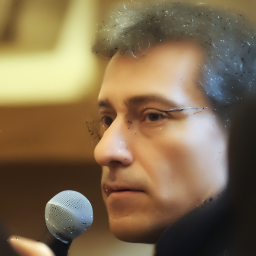}}
 & \raisebox{-25pt}{\includegraphics[width = 0.118\textwidth]{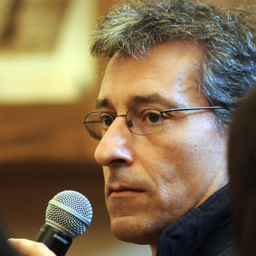}}\\ 
 & \small{Input} & \small{L2TV} & \small{DDNM$+$} & \small{$\Pi$GDM} 
 & \small{DPS} & \small{DMPS} & 
 \small{Ours} 
 & \small{GT}\\
  \end{tabular}
  \caption{Visualization of the results from different methods for Gaussian noisy linear problems on the FFHQ $256\times 256$. Gaussian noise ($\sigma_y=0.05$) is added to the measurement. \label{fig:ffhq_gauss}} %
 \end{figure}
 
 \begin{figure}[!htbp]
   \centering 
  \begin{tabular}{c@{\hspace*{1pt}}c@{\hspace*{1pt}}c@{\hspace*{1pt}}c@{\hspace*{1pt}}c@{\hspace*{1pt}}c@{\hspace*{1pt}}c@{\hspace*{1pt}}c@{\hspace*{1pt}}c@{\hspace*{1pt}}c@{\hspace*{2pt}}c@{\hspace*{2pt}}c@{\hspace*{2pt}}c@{\hspace*{1pt}}}
   \rotatebox[origin=c]{90}{\small{SR}} &\raisebox{-25pt}{\includegraphics[width = 0.118\textwidth]{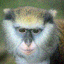}}
 & \raisebox{-25pt}{\includegraphics[width = 0.118\textwidth]{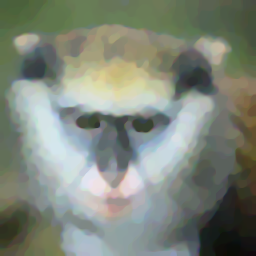}}
 & \raisebox{-25pt}{\includegraphics[width = 0.118\textwidth]{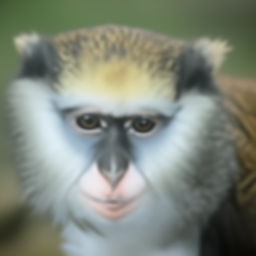}}
 & \raisebox{-25pt}{\includegraphics[width = 0.118\textwidth]{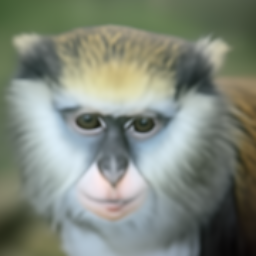}}
 & \raisebox{-25pt}{\includegraphics[width = 0.118\textwidth]{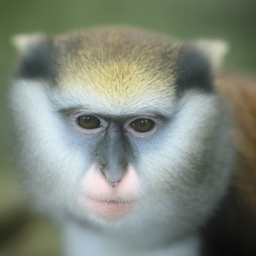}}
 & \raisebox{-25pt}{\includegraphics[width = 0.118\textwidth]{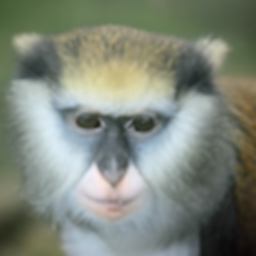}}
 & \raisebox{-25pt}{\includegraphics[width = 0.118\textwidth]{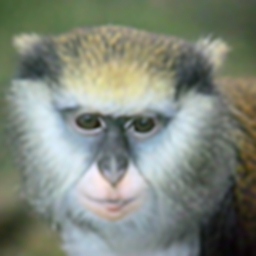}}
    & \raisebox{-25pt}{\includegraphics[width = 0.118\textwidth]{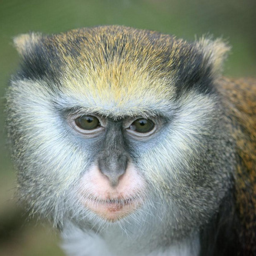}}\\
    \rotatebox[origin=c]{90}{\small{Deblur (Gau)}} &   \raisebox{-25pt}{\includegraphics[width = 0.118\textwidth]{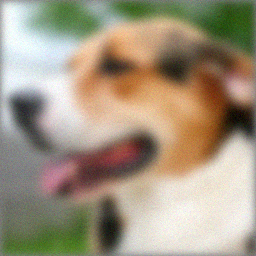}}
 & \raisebox{-25pt}{\includegraphics[width = 0.118\textwidth]{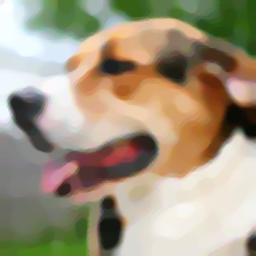}}
 & \raisebox{-25pt}{\includegraphics[width = 0.118\textwidth]{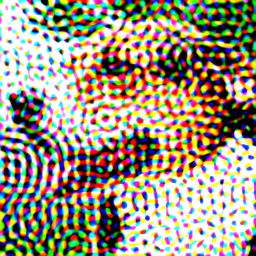}}
 & \raisebox{-25pt}{\includegraphics[width = 0.118\textwidth]{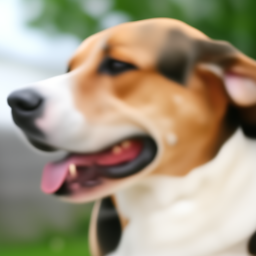}}
 & \raisebox{-25pt}{\includegraphics[width = 0.118\textwidth]{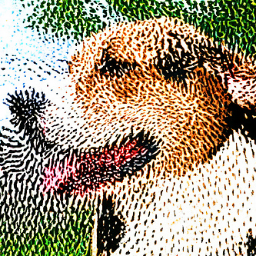}}
 & \raisebox{-25pt}{\includegraphics[width = 0.118\textwidth]{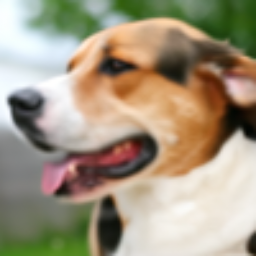}}
 & \raisebox{-25pt}{\includegraphics[width = 0.118\textwidth]{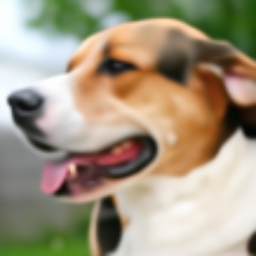}}
    & \raisebox{-25pt}{\includegraphics[width = 0.118\textwidth]{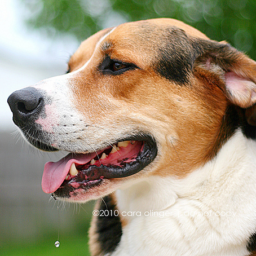}}\\
    \rotatebox[origin=c]{90}{\small{Deblur (uni)}} &\raisebox{-25pt}{\includegraphics[width = 0.118\textwidth]{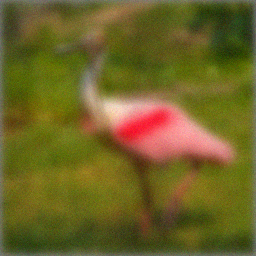}}
 & \raisebox{-25pt}{\includegraphics[width = 0.118\textwidth]{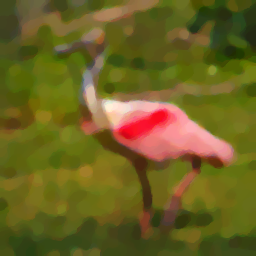}}
 & \raisebox{-25pt}{\includegraphics[width = 0.118\textwidth]{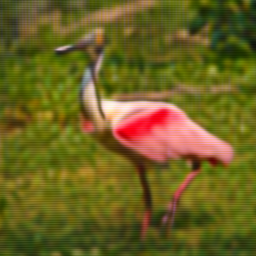}}
 & \raisebox{-25pt}{\includegraphics[width = 0.118\textwidth]{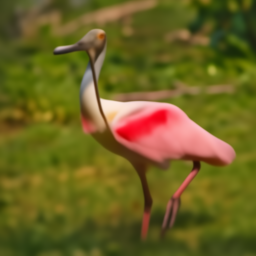}}
 & \raisebox{-25pt}{\includegraphics[width = 0.118\textwidth]{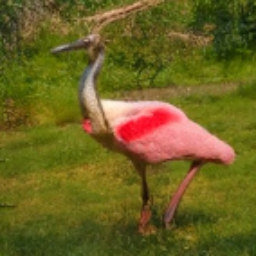}}
 & \raisebox{-25pt}{\includegraphics[width = 0.118\textwidth]{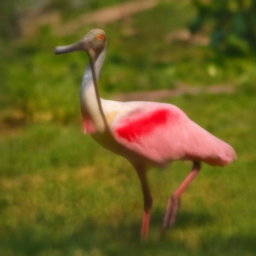}}
 & \raisebox{-25pt}{\includegraphics[width = 0.118\textwidth]{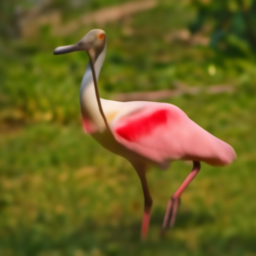}}
    & \raisebox{-25pt}{\includegraphics[width = 0.118\textwidth]{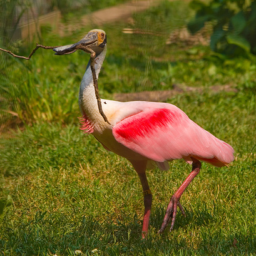}}\\
    \rotatebox[origin=c]{90}{\small{Inp (box)}} &   \raisebox{-25pt}{\includegraphics[width = 0.118\textwidth]{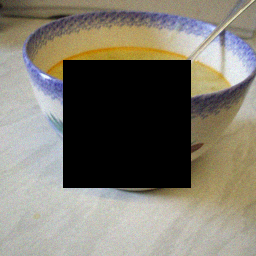}}
 & \raisebox{-25pt}{\includegraphics[width = 0.118\textwidth]{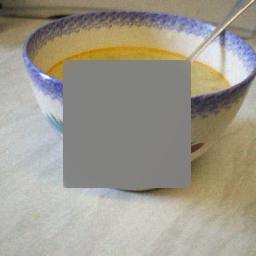}}
 & \raisebox{-25pt}{\includegraphics[width = 0.118\textwidth]{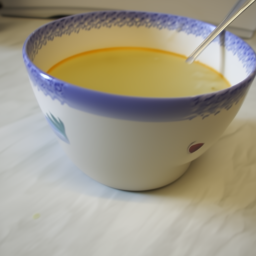}}
 & \raisebox{-25pt}{\includegraphics[width = 0.118\textwidth]{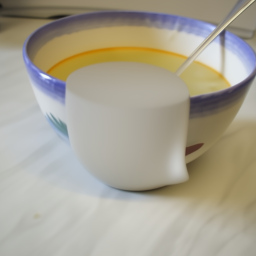}}
 & \raisebox{-25pt}{\includegraphics[width = 0.118\textwidth]{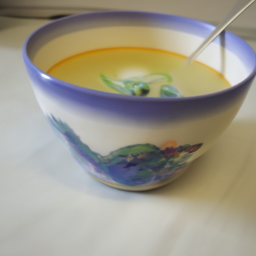}}
 & \raisebox{-25pt}{\includegraphics[width = 0.118\textwidth]{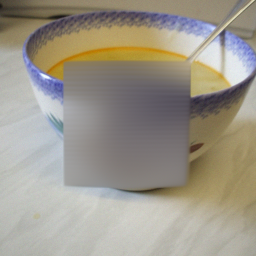}}
 & \raisebox{-25pt}{\includegraphics[width = 0.118\textwidth]{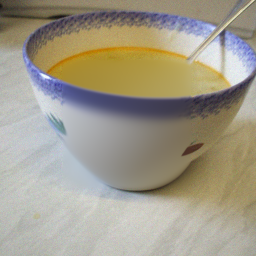}}
    & \raisebox{-25pt}{\includegraphics[width = 0.118\textwidth]{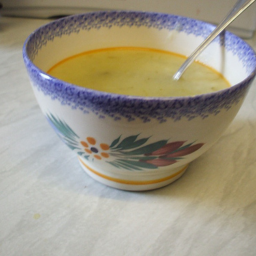}}\\
    \rotatebox[origin=c]{90}{\small{Inp (rand.)}} &      \raisebox{-25pt}{\includegraphics[width = 0.118\textwidth]{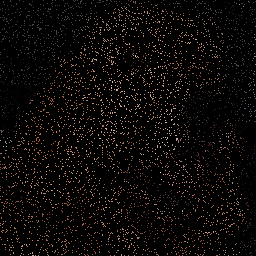}}
 & \raisebox{-25pt}{\includegraphics[width = 0.118\textwidth]{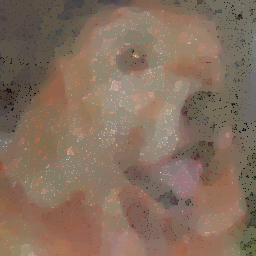}}
 & \raisebox{-25pt}{\includegraphics[width = 0.118\textwidth]{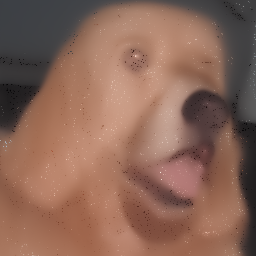}}
 & \raisebox{-25pt}{\includegraphics[width = 0.118\textwidth]{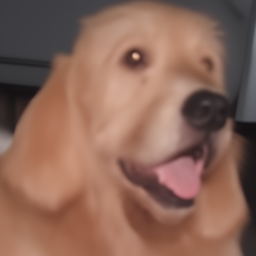}}
 & \raisebox{-25pt}{\includegraphics[width = 0.118\textwidth]{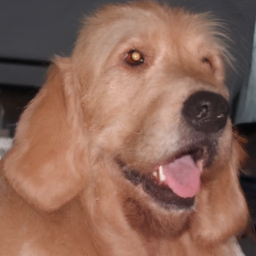}}
 & \raisebox{-25pt}{\includegraphics[width = 0.118\textwidth]{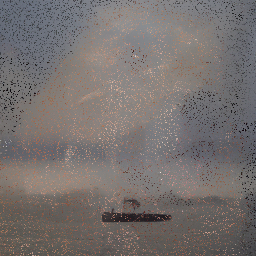}}
 & \raisebox{-25pt}{\includegraphics[width = 0.118\textwidth]{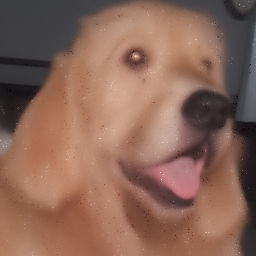}}
 & \raisebox{-25pt}{\includegraphics[width = 0.118\textwidth]{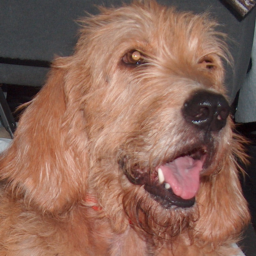}}\\
 & \small{input} & \small{L2TV} & \small{DDNM$+$} & \small{$\Pi$GDM} 
 & \small{DPS} & \small{DMPS} & 
 \small{Ours} 
 & \small{GT}\\
  \end{tabular}%
 \caption{Visualization of the results from different methods for Gaussian noisy linear problems on the ImageNet $256\times 256$. Gaussian noise ($\sigma_y=0.05$) is added to the measurement. \label{fig:img_gauss}} %
\end{figure}

\subsection{Experiments for SVD-free linear operators}

In this part, we present the performance of our methods for the SVD-free linear operators.  To obtain the SVD for general operators is a bottleneck for some complicated linear problems, hence DDRM and DDNM$+$ are excluded for comparison. We run 20-step CG to compute the multiplication of matrix inverse and a vector in~\eqref{eq:ours}. In the case of motion deblurring, a $61\times 61$ motion kernel is considered, as its SVD can be not implemented efficiently like blurring with a separable Gaussian kernel. We compare Ours to DPS~\cite{chung2022diffusion}, $\Pi$GDM and Meng's DMPS approach. See Table~\ref{tab:3} for the quantitative results and Figure~\ref{fig:motion_deblur} for the visualization. Notice that with fixed other hyperparameters, following the $\Pi$GDM, the results are noise-corrupted, hence we exclude it for the quantitative comparison.

\begin{table}[!htp]
   \centering
  \caption{Quantitative results for noisy motion deblurring (additive noise of $\sigma_y=0.05$) on the FFHQ and ImageNet dataset. \label{tab:3}}
\begin{tabular}{@{\hspace{0pt}}c@{\hspace{3pt}}c@{\hspace{3pt}}c@{\hspace{3pt}}c@{\hspace{3pt}}c@{\hspace{2pt}}c@{\hspace{3pt}}c@{\hspace{3pt}}c@{\hspace{3pt}}c@{\hspace{3pt}}cc@{\hspace{3pt}}c@{\hspace{3pt}}c@{\hspace{3pt}}c@{\hspace{2pt}}cc@{\hspace{3pt}}c@{\hspace{3pt}}c@{\hspace{3pt}}c@{\hspace{2pt}}c@{\hspace{3pt}}c@{\hspace{3pt}}c@{\hspace{0pt}}cc@{\hspace{3pt}}c@{\hspace{3pt}}c@{\hspace{0pt}}}
 \toprule
    &   \multicolumn{3}{c}{\textbf{FFHQ}}  && \multicolumn{3}{c}{\textbf{ImageNet}} \\
        \cline{2-4} \cline{6-8}
     Method & PSNR  & SSIM & LPIPS && PSNR  & SSIM & LPIPS \\
   \toprule
  DMPS   & \ul{30.51} & \ul{0.837} & \ul{0.232} &  & \ul{28.26} & \ul{0.782} & \ul{0.269} \\
  Ours   & \bl{32.30} & \bl{0.879} & \bl{0.142} &  & \bl{29.81} & \bl{0.810} & \bl{0.179} \\
   \midrule
   DPS       & 25.94 & 0.711 & 0.233 &  & 23.44 & 0.587 & 0.361 \\
\toprule
\end{tabular} %
\end{table}

\begin{figure}[!htbp]
   \centering 
  \begin{tabular}{c@{\hspace*{0pt}}c@{\hspace*{0pt}}c@{\hspace*{0pt}}c@{\hspace*{0pt}}c@{\hspace*{0pt}}c@{\hspace*{0pt}}c@{\hspace*{0pt}}c@{\hspace*{0pt}}c@{\hspace*{0pt}}c@{\hspace*{0pt}}c@{\hspace*{2pt}}c@{\hspace*{2pt}}c@{\hspace*{0pt}}}
 \includegraphics[width = 0.12\textwidth]{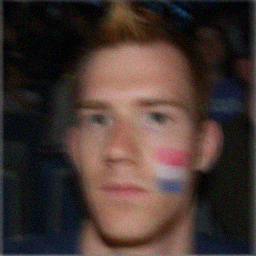}
    & \includegraphics[width = 0.12\textwidth]{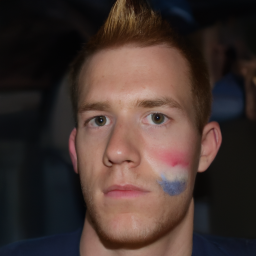}
  & \includegraphics[width = 0.12\textwidth]{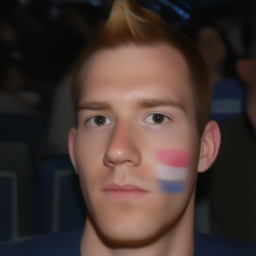}
 & \includegraphics[width = 0.12\textwidth]{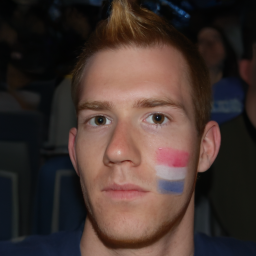}
 & \includegraphics[width = 0.12\textwidth]{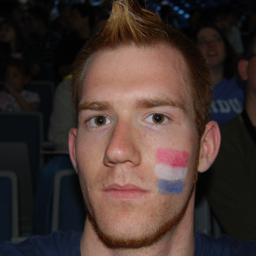}\\[-2pt]
 \includegraphics[width = 0.12\textwidth]{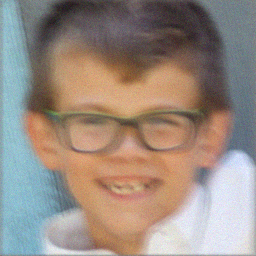}
    & \includegraphics[width = 0.12\textwidth]{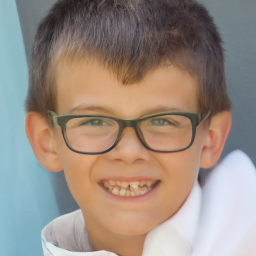}
       & \includegraphics[width = 0.12\textwidth]{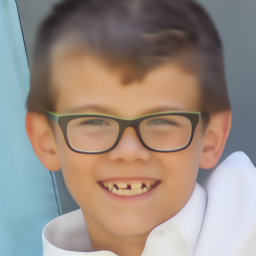}
 & \includegraphics[width = 0.12\textwidth]{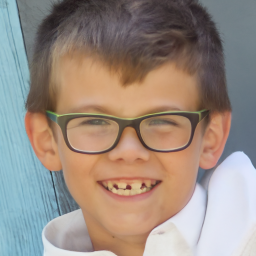}
 & \includegraphics[width = 0.12\textwidth]{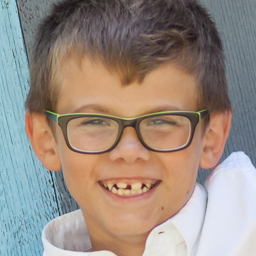}\\
    \tiny{Input} & \tiny{DPS} & \tiny{DMPS} & \tiny{Ours} & \tiny{GT} \\
 \end{tabular}
 \caption{Results for noisy motion deblurring.\label{fig:motion_deblur}} %
 \end{figure}
We also evaluate our method to the sparse-view CT application. Aligned with the available generative model, the VE SDE-based version~\eqref{eq:vesde} of our method is adopted. We simulate the CT measurement process with parallel beam geometry with evenly-spaced 180 degrees~\cite{chung2022improving}. Two configurations with $18$ and $30$ total views are tested. We compare our methods to the MCG~\cite{chung2022improving}, Score-based SDE~\cite{song2021solving}, and the traditional analytic FBP reconstruction. We consider both noiseless and noise cases. For the noisy case, \ie, Gaussian noise with $\sigma_y=3.0$ (corresponding to SNR$=24$dB) is added to the sinogram data. All compared diffusion methods are performed with 1000 diffusion steps. See Table~\ref{tab:ct} for the quantitative results on the 100 images from AAPM dataset. See Figure~\ref{fig:svct} for the visualization of the reconstruction from 18/30 views respectively. The projection step of MCG introduces noise in the results for noisy measurements. As a comparison to Score-SDE and MCG, our method yields the best or comparable  results with the least computing time. 

\begin{table}[!htp]
  \centering
  \caption{Quantitative results for noiseless and noisy sparse-view CT. \label{tab:ct}}
 \begin{tabular}{@{\hspace{0pt}}c@{\hspace{3pt}}c@{\hspace{3pt}}c@{\hspace{3pt}}c@{\hspace{3pt}}c@{\hspace{3pt}}c@{\hspace{3pt}}c@{\hspace{3pt}}c@{\hspace{3pt}}c@{\hspace{3pt}}c@{\hspace{3pt}}c@{\hspace{3pt}}c@{\hspace{3pt}}c@{\hspace{3pt}}c@{\hspace{3pt}}c@{\hspace{2pt}}c@{\hspace{3pt}}c@{\hspace{3pt}}c@{\hspace{3pt}}c@{\hspace{3pt}}c@{\hspace{2pt}}c@{\hspace{3pt}}c@{\hspace{3pt}}c@{\hspace{0pt}}c@{\hspace{3pt}}c@{\hspace{3pt}}c@{\hspace{3pt}}c@{\hspace{0pt}}}
   \toprule
  \multirow{3}{*}{Method}  & \multicolumn{5}{c}{\textbf{noiseless}} && \multicolumn{5}{c}{\textbf{noisy ($\sigma_y=3.0$)}} & \multirow{3}{*}{Time (s)} \\
   \cline{2-6} \cline{8-12}
    &   \multicolumn{2}{c}{\textbf{18-view}}  && \multicolumn{2}{c}{\textbf{30-view}} &&  \multicolumn{2}{c}{\textbf{18-view}}  && \multicolumn{2}{c}{\textbf{30-view}}\\
        \cline{2-3} \cline{5-6} \cline{8-9} \cline{11-12}
 & PSNR  & SSIM && PSNR  & SSIM && PSNR  & SSIM && PSNR  & SSIM &\\
   \toprule
   Score-SDE & 21.55 & 0.791 &  & 23.98 & 0.873 &  & 21.51 & 0.794 &  & 24.07 & 0.875 & 506 \\
   MCG       & 36.60  & 0.908 &  & 37.16 & 0.912 &  & 27.21 & 0.666 &  & 29.12 & 0.747 & 1270 \\
   Ours      & 36.14 & 0.921 &  & 37.34 & 0.93  &  & 29.37 & 0.868 &  & 29.51 & 0.869 & 241  \\
   \bottomrule
\end{tabular}
\end{table}

\begin{figure}[!htbp]
   \centering 
  \begin{tabular}{c@{\hspace*{1pt}}c@{\hspace*{1pt}}c@{\hspace*{1pt}}c@{\hspace*{1pt}}c@{\hspace*{1pt}}c@{\hspace*{1pt}}c@{\hspace*{1pt}}c@{\hspace*{1pt}}c@{\hspace*{1pt}}c@{\hspace*{2pt}}c@{\hspace*{2pt}}c@{\hspace*{2pt}}c@{\hspace*{1pt}}}
 \rotatebox[origin=c]{0}{\textbf{(a)}} &\raisebox{-0.5\height}{\includegraphics[width = 0.12\textwidth]{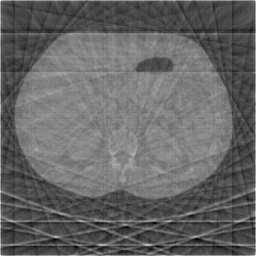}}
 & \raisebox{-0.5\height}{\includegraphics[width = 0.12\textwidth]{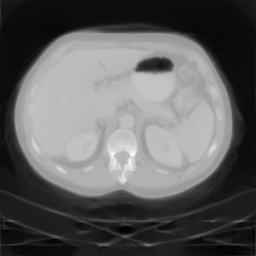}}
 & \raisebox{-0.5\height}{\includegraphics[width = 0.12\textwidth]{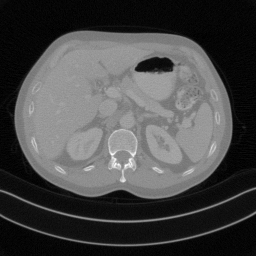}}
 & \raisebox{-0.5\height}{\includegraphics[width = 0.12\textwidth]{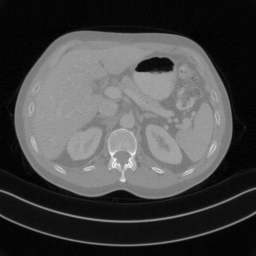}}
 & \raisebox{-0.5\height}{\includegraphics[width = 0.12\textwidth]{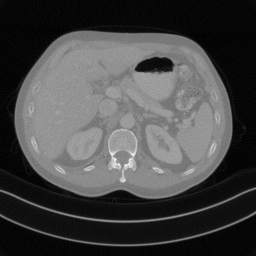}}\\[-2pt]
 \rotatebox[origin=c]{0}{\textbf{(b)}} &\raisebox{-0.5\height}{\includegraphics[width = 0.12\textwidth]{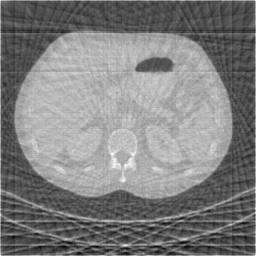}}
 & \raisebox{-0.5\height}{\includegraphics[width = 0.12\textwidth]{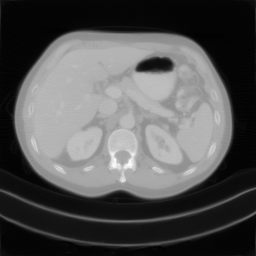}}
 & \raisebox{-0.5\height}{\includegraphics[width = 0.12\textwidth]{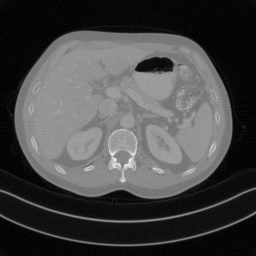}}
 & \raisebox{-0.5\height}{\includegraphics[width = 0.12\textwidth]{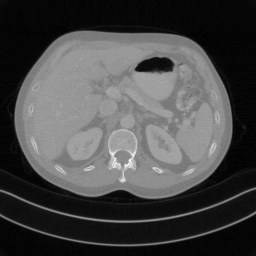}}
 & \raisebox{-0.5\height}{\includegraphics[width = 0.12\textwidth]{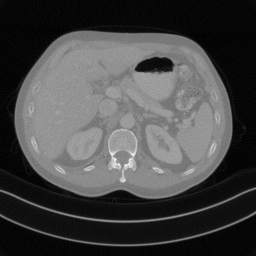}}\\[-2pt]
 \rotatebox[origin=c]{0}{\textbf{(c)}} &\raisebox{-0.5\height}{\includegraphics[width = 0.12\textwidth]{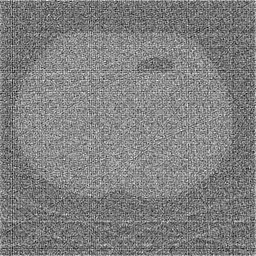}}
 & \raisebox{-0.5\height}{\includegraphics[width = 0.12\textwidth]{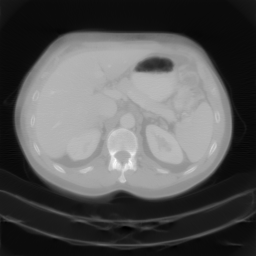}}
 & \raisebox{-0.5\height}{\includegraphics[width = 0.12\textwidth]{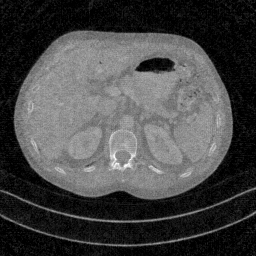}}
 & \raisebox{-0.5\height}{\includegraphics[width = 0.12\textwidth]{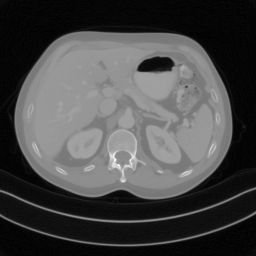}}
 & \raisebox{-0.5\height}{\includegraphics[width = 0.12\textwidth]{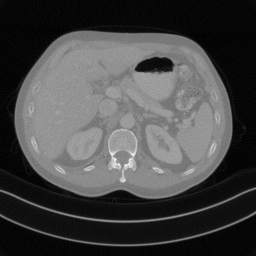}}\\[-2pt]
 \rotatebox[origin=c]{0}{\textbf{(d)}} &\raisebox{-0.5\height}{\includegraphics[width = 0.12\textwidth]{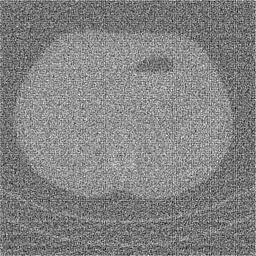}}
 & \raisebox{-0.5\height}{\includegraphics[width = 0.12\textwidth]{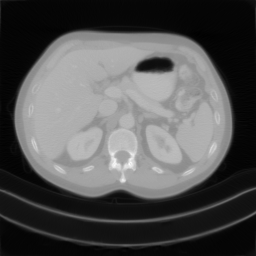}}
 & \raisebox{-0.5\height}{\includegraphics[width = 0.12\textwidth]{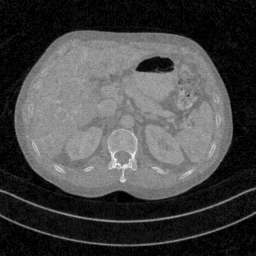}}
 & \raisebox{-0.5\height}{\includegraphics[width = 0.12\textwidth]{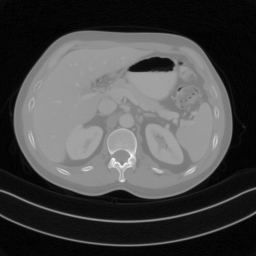}}
 & \raisebox{-0.5\height}{\includegraphics[width = 0.12\textwidth]{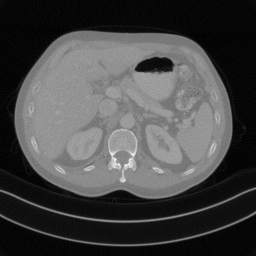}}\\
 &   \small{FBP} & \small{Score} & \small{MCG} & \small{Ours} 
 & \small{Ref}\\
 \end{tabular} 
 \caption{The results for 18/30-view CT reconstruction. The four rows corresponds to four cases. (a) and (b) 18/30-view noiseless measurements respectively, while (c) and (d) 18/30-view noisy measurements ($\sigma_y=3.0$) respectively. \label{fig:svct}} %
 \end{figure}
 
\subsection{JPEG decompression experiment}

We evaluate our method for the JPEG decompression task as well. We set the quality factor of JPEG compressing to 5,10 respectively and Gaussian noise with $\sigma_y=0.05$ is added to the JPEG compressed image. To the best of our knowledge, our method is the first work to address Gaussian noisy JPEG decompression. The works~\cite{kawar2022jpeg,song2022pseudoinverse} only consider noiseless case. Figure~\ref{fig:jpeg} shows the results of the reconstruction from noisy JPEG compression.

\begin{figure}[!htbp]
   \centering 
  \begin{tabular}{c@{\hspace*{0pt}}c@{\hspace*{0pt}}c@{\hspace*{1pt}}c@{\hspace*{0pt}}c@{\hspace*{0pt}}c@{\hspace*{0pt}}c@{\hspace*{0pt}}c@{\hspace*{0pt}}c@{\hspace*{0pt}}c@{\hspace*{0pt}}c@{\hspace*{2pt}}c@{\hspace*{2pt}}c@{\hspace*{0pt}}}
    \includegraphics[width = 0.12\textwidth]{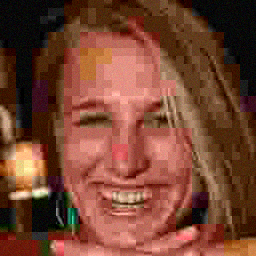}
     & \includegraphics[width = 0.12\textwidth]{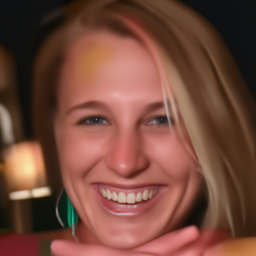}
 & \includegraphics[width = 0.12\textwidth]{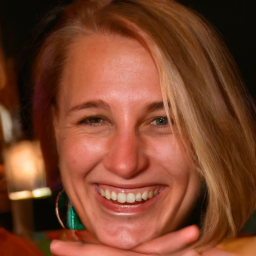}
    & \includegraphics[width = 0.12\textwidth]{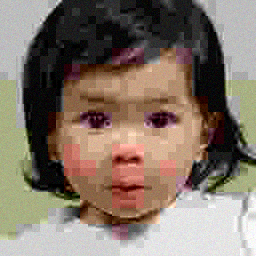}
       & \includegraphics[width = 0.12\textwidth]{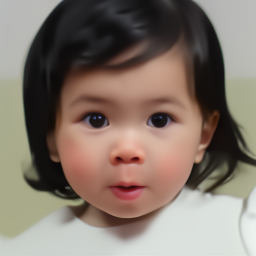}
 & \includegraphics[width = 0.12\textwidth]{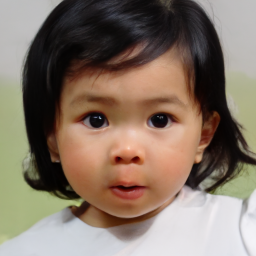}\\
    \includegraphics[width = 0.12\textwidth]{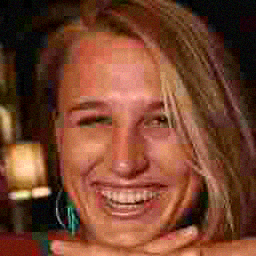}
     & \includegraphics[width = 0.12\textwidth]{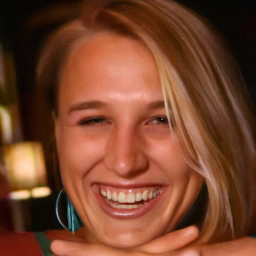}
 & \includegraphics[width = 0.12\textwidth]{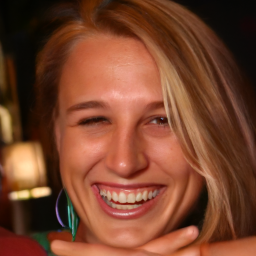}
    & \includegraphics[width = 0.12\textwidth]{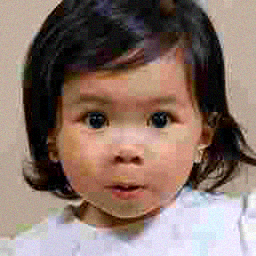}
       & \includegraphics[width = 0.12\textwidth]{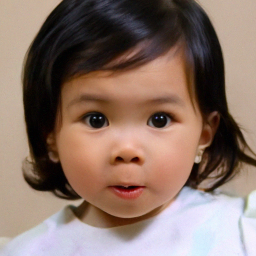}
 & \includegraphics[width = 0.12\textwidth]{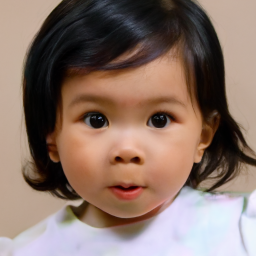}\\
    \tiny{Input} & \tiny{DMPS} & \tiny{Ours} & \tiny{Input} & \tiny{DMPS} & \tiny{Ours}\\
 \end{tabular}
 \caption{Results for noisy JPEG decompression. Top/bottom are results for QF=5/10 respectively.\label{fig:jpeg}} %
\end{figure}

\section{Conclusion}\label{sec:conclusion}

The diffusion posterior sampling framework is a new approach to inverse problem solving, where the diffusion model serves as the image prior. The existing measurement guidance step employs a gradient descent with respect to a noisy likelihood approximation. Motivated by the connection between the two expectations of $p(\bm{x}_0|\bm{x}_t)$ and $p(\bm{x}_0|\bm{x}_t,\bm{y})$, we derive an MAP-based proxy solution to compute the expectation of $p(\bm{x}_0|\bm{x}_t,\bm{y})$. With this approach, the method is more flexible and avoids the backpropagation of existing DPS-like methods, which saves the computation cost and reduces the memory consumption. Unlike the SVD-dependent works~\cite{kawar2022denoising,wang2022zero}, our method allows a broad range of linear operators without the necessity of a full SVD implementation. Comparing to highly related work~\cite{meng2022diffusion}, which assumed a crude uninformative prior $p(\bm{x}_0)$, our approximation contains an additional extra, which is crucial for inpainting tasks. 
\section*{Acknowledgments}
We would like to acknowledge the authors' affiliated institution for providing the computing environment that made this study possible.

\bibliographystyle{siamplain}

\end{document}